\documentclass{amsart}

\RequirePackage{xcolor}
\definecolor{mg}{RGB}{34, 139, 34}

\RequirePackage{float}
\RequirePackage[margin=1in]{geometry}
\RequirePackage{graphicx}
\RequirePackage{amsmath, amsthm, amssymb, amsfonts,slashed}
\usepackage{aliascnt}
\RequirePackage{xifthen}
\RequirePackage{indentfirst}
\RequirePackage{microtype}

\newenvironment{Eq}[1][]{\begin{equation}\ifthenelse{\equal{#1}{}}{}{\tag{#1}}\begin{aligned}}{\end{aligned}\end{equation}\ignorespacesafterend}
\newenvironment{Eq*}[1][]{\begin{equation*}\ifthenelse{\equal{#1}{}}{}{\tag{#1}}\begin{aligned}}{\end{aligned}\end{equation*}\ignorespacesafterend}
\numberwithin{equation}{section}
\newtheorem{theorem}{Theorem}[section]
\newtheorem{hypothesis}{Hypothesis}

\newaliascnt{proposition}{theorem}
\newtheorem{proposition}[proposition]{Proposition}
\aliascntresetthe{proposition}

\newaliascnt{corollary}{theorem}

\aliascntresetthe{corollary}

\newaliascnt{lemma}{theorem}
\newtheorem{lemma}[lemma]{Lemma}
\aliascntresetthe{lemma}

\newaliascnt{claim}{theorem}
\newtheorem{claim}[claim]{Claim}%
\aliascntresetthe{claim}

\theoremstyle{remark}%
\newtheorem{remark}{Remark}[section] 
\theoremstyle{definition}

\usepackage[colorlinks=true,linkcolor=red,citecolor=blue,urlcolor=blue]{hyperref}
\usepackage[capitalise,noabbrev]{cleveref}

\def\<{\langle}             \def\>{\rangle}

\newcommand{\varE}{{\mathfrak E}}
\newcommand{\varF}{\varE'}
\newcommand{\varL}{\mathcal L}
\DeclareMathOperator{\supp}{supp}
\DeclareMathOperator{\fd}{d}	\renewcommand{\d}{\fd}

\newcommand{\tg}{{\tilde g}}
\newcommand{\tpartial}{{\tilde \partial}}
\DeclareMathOperator{\II}{i}

\newcommand{\bphi}{{\boldsymbol{\phi}}} 
\newcommand{\bH}{{\mathcal H}}

\newcommand{\tw}{{\tilde{w}}}
\newcommand{\tq}{{\tilde{q}}}
\newcommand{\tr}{{\tilde{r}}}
\newcommand{\tmu}{{\tilde{\mu}}}

\newcommand{\rR}{{\mathbb{R}}}
\newcommand{\rH}{{\mathbb H}}
\newcommand{\rS}{{\mathbb{S}}}
\newcommand{\rN}{{\mathbb{N}}}
\newcommand{\spartial}{{\slashed \partial}}
\newcommand{\hpartial}{{\hat \partial}}

\newcommand{\kl}[1]{\mathopen{}\left#1}
\newcommand{\kr}[1]{\right#1}
\newcommand{\Low}[1]{\limits_{\substack{#1}}}

\crefname{equation}{}{}
\crefname{theorem}{Theorem}{Theorems}
\Crefname{theorem}{Theorem}{Theorems}
\crefname{proposition}{Proposition}{Propositions}
\Crefname{proposition}{Proposition}{Propositions}
\crefname{corollary}{Corollary}{Corollaries}
\Crefname{corollary}{Corollary}{Corollaries}
\crefname{lemma}{Lemma}{Lemmas}
\Crefname{lemma}{Lemma}{Lemmas}
\crefname{claim}{Claim}{Claims}
\Crefname{claim}{Claim}{Claims}
\crefname{hypothesis}{Hypothesis}{Hypotheses}
\Crefname{hypothesis}{Hypothesis}{Hypotheses}
\crefname{remark}{Remark}{Remarks}
\Crefname{remark}{Remark}{Remarks}
\crefname{definition}{Definition}{Definitions}
\Crefname{definition}{Definition}{Definitions}

\title[Strichartz estimates for Klein-Gordon equations on asymptotically hyperbolic manifolds]
{Strichartz estimates for Klein-Gordon equations on asymptotically hyperbolic manifolds with applications}

\author{Wei Dai}
\address{
School of Mathematical Sciences\\
 Zhejiang University of Technology\\ 
 Hangzhou 310023,
 China}
\email{daiw23@zjut.edu.cn}

\author{Chengbo Wang$^{*}$}\thanks{* Corresponding author}
\address{School of Mathematical Sciences\\
Zhejiang Normal University\\
Jinhua 321004, China}
\email{wangcbo@zjnu.edu.cn}

\thanks{The first author was supported in part by NSFC 12501292. The second author was supported in part by NSFC 12671293.}

\date{\today}

\begin{document}
\bibliographystyle{plain}

\begin{abstract}
In this paper, we establish global-in-time Strichartz estimates for Klein-Gordon equations on non-trapping asymptotically hyperbolic manifolds. Our result is a positive-mass counterpart of the global Strichartz estimates of Sire, Sogge, Wang, and Zhang for shifted wave equations, and recovers the same range of admissible pairs. The main new feature is the low-frequency analysis, where the positive mass changes the spectral phase and leads to a different dispersive mechanism. Combining the resulting low-frequency estimates with microlocalized high-frequency estimates yields the full global Strichartz estimates. As applications, we obtain small-data global existence and scattering for power-type and certain derivative semilinear equations, as well as small-data global existence for a class of smooth semilinear equations without imposing a null condition.
\end{abstract}

\keywords{Klein-Gordon equations, asymptotically hyperbolic manifolds, Strichartz estimates, local energy estimates, semilinear wave equations}

\subjclass[2020]{35L71, 35L05, 35B40, 58J45}

\maketitle

\section{Introduction}\label{sec:introduction}
The purpose of this paper is to study the dispersive behavior of a Klein-Gordon equation on non-trapping asymptotically hyperbolic manifolds, and its applications to small-data semilinear Cauchy problems.

\subsection{Background on Strichartz estimates}
Strichartz estimates play a central role in the study of nonlinear wave, Klein-Gordon, and Schr\"odinger equations; see, for instance, Tao \cite{MR2233925} and the references therein.
In this direction, our main linear goal is to prove global-in-time Strichartz estimates for a Klein-Gordon equation in the variable-curvature asymptotically hyperbolic setting.

More precisely, let $(M^\circ,g)$ be a complete noncompact Riemannian manifold whose metric is asymptotically hyperbolic and non-trapping in the sense recalled below.  We prove global-in-time Strichartz estimates for solutions to
\begin{Eq*}
\kl(\partial_t^2-\Delta_g-\frac{(n-1)^2}{4}+\kappa^2\kr)\phi=F,\qquad \kappa>0,
\end{Eq*}
where $\Delta_g$ is the Laplace-Beltrami operator on $(M^\circ,g)$.

Our motivation comes from the work of Sire, Sogge, the second author, and Zhang \cite{MR4169670} on the corresponding shifted wave equation without the positive mass term.  They proved global-in-time Strichartz estimates for
\begin{Eq*}
\partial_t^2u-\Delta_g u-\frac{(n-1)^2}{4}u=F
\end{Eq*}
on non-trapping asymptotically hyperbolic manifolds, and applied them to the Strauss-type problem.  Our main estimate can be viewed as a positive-mass analogue of that result.

The positive mass term changes the low-frequency part of the analysis in an essential way.  For the shifted wave equation, the phase is linear in the spectral parameter $\lambda$.  In the present positive-mass case, the phase is instead $\sqrt{\lambda^2+\kappa^2}$, which produces a different low-frequency decay mechanism.  The high-frequency analysis, on the other hand, remains close in spirit to \cite{MR4169670}: one combines microlocalized spectral measure estimates with a low-high frequency decomposition of the propagator.

Before stating the estimates precisely, we recall the geometric setting and the linear equation under consideration.

\subsection{Geometric and linear setting}
We work on an $n$-dimensional complete noncompact Riemannian manifold $(M^\circ,g)$ which is the interior of a smooth compact manifold with boundary $M$.  Let $x\in C^\infty(M)$ be a boundary defining function, namely $x\geq0$, $\partial M=\{x=0\}$, and $\d x\neq0$ on $\partial M$.  The metric $g$ is called conformally compact if $x^2g$ extends smoothly to $M$ as a Riemannian metric.  Following the standard convention, we say that $g$ is asymptotically hyperbolic if, in addition,
\begin{Eq*}
|dx|_{x^2g}=1\quad\text{on }\partial M .
\end{Eq*}
Equivalently, after choosing $x$ geodesic near the boundary, the metric takes the form
\begin{Eq}
g=\frac{\d x^2+ h_{jk}(x,y)\d y^j\d y^k}{x^2}
\end{Eq}
Here and throughout the paper, we use the Einstein summation convention.  The expression above is written on a collar neighborhood $[0,\varepsilon)_x\times\partial M$, where $y=(y_1,\ldots,y_{n-1})$ are local coordinates on $\partial M$ and $h(x,y)=h_{jk}(x,y)\d y^j\d y^k$ is a smooth family of metrics on $\partial M$.  Mazzeo \cite{MR961517} showed that the sectional curvature of such a metric tends to $-1$ as $x\to0$.

Such manifolds have been studied in the works of Mazzeo-Melrose \cite{MR916753}, Chen-Hassell \cite{MR3473907, MR3805767}, Sire-Sogge-Wang-Zhang \cite{MR4169670, MR4432952}, and others.  We further assume that $M$ is non-trapping, meaning that every geodesic in $M^\circ$ reaches $\partial M$ in both forward and backward time directions.  A canonical example is hyperbolic space $\rH^n=(B_1,\eta)$, where $B_1$ is the unit ball and
\begin{Eq*}
\eta=\frac{\d x^2+\frac14(1-x^2)^2\d\Theta^2}{x^2}\quad\text{on}\quad (0,1)\times \rS^{n-1},
\end{Eq*}
with $\d\Theta^2$ denoting the standard metric on $\rS^{n-1}$.

As shown in \cite{MR916753}, the continuous spectrum of $-\Delta_g$ is contained in $[(n-1)^2/4, +\infty)$,
while its point spectrum lies in $(0, (n-1)^2/4)$.
Set
\begin{Eq*}
H_0:=-\Delta_g-\frac{(n-1)^2}{4},\qquad H_\kappa:=H_0+\kappa^2.
\end{Eq*}
As in \cite{MR4169670}, we assume that $-\Delta_g$ has no $L^2$ eigenvalues below $\frac{(n-1)^2}{4}$ and that zero is not a resonance of $H_0$; the notion of resonance is understood in the standard resolvent sense, see \cite{MR916753}.
We consider the following Klein-Gordon equation
\begin{Eq}\label{eq:positive-mass-wave-equation}
\begin{cases}
(\partial_t^2-\Delta_g-\frac{(n-1)^2}{4}+\kappa^2)\phi(t,z)=F(t,z),\qquad (t,z)\in \rR_+\times M^\circ\\
\phi(0,z)=\psi_0(z),\quad \partial_t\phi(0,z)=\psi_1(z),\qquad \kappa>0.
\end{cases}
\end{Eq}

In the case where $(M^\circ, g) = (B_1, \eta) = \rH^n$,
equation \eqref{eq:positive-mass-wave-equation} can be compared with the Klein-Gordon equation in Euclidean space.
In particular, when $\kappa = 0$, the shifted operator $H_0=-\Delta_g-\frac{(n-1)^2}{4}$ has continuous spectrum starting at zero.
In this spectral sense, the shifted wave equation plays the role of the standard wave equation in the asymptotically hyperbolic setting.
However, as shown by Anker-Pierfelice-Vallarino \cite{MR2902129},
the richer geometric structure of hyperbolic space allows for a broader range of admissible Strichartz pairs
than those available for the wave equation in Euclidean space.
Earlier related results can be found in the works of Fontaine \cite{MR1302795, MR1443052}, Ionescu \cite{MR1767376}, and the references therein.
As for the case $\kappa > 0$, Strichartz estimates analogous to those in \cite{MR2902129} were established by Anker-Pierfelice \cite{MR3254350}.
Earlier contributions in this direction also appear in the work of Metcalfe-Taylor \cite{MR2775816, MR2944727} and references therein.
For the shifted wave equation in the variable-curvature asymptotically hyperbolic setting, the relevant result we take as a model is due to Sire-Sogge-Wang-Zhang \cite{MR4169670}.
They established Strichartz estimates for the case $\kappa=0$, similar to those in \cite{MR2902129},
with the exception of certain critical cases of admissible pairs.

Recall that, for the Euclidean wave equation,
a common approach to proving Strichartz estimates is to first derive single-frequency estimates using $L^2$ bounds and dispersive estimates,
and then extend these to general solutions via Littlewood-Paley theory.
However, in the hyperbolic setting, the Euclidean Littlewood-Paley theory is not directly available, and spectral multipliers require additional analysis;
see Bouclet \cite{MR2638890}.
The strategy in \cite{MR3254350} and \cite{MR2902129} for overcoming this obstacle relies on precise estimates of the fundamental solution,
obtained through the Fourier transform on hyperbolic space.
This approach, however, breaks down in the presence of variable curvature.
To circumvent this issue, the authors of \cite{MR4169670} adopted a spectral method:
they decomposed the solution operator into low- and high-frequency components,
and derived the desired estimates by analyzing each part separately.

In this paper, we adopt the general strategy of \cite{MR4169670},
carrying out a low-high frequency decomposition of the solution operator tailored to equation \eqref{eq:positive-mass-wave-equation},
and establishing the corresponding estimates.
We emphasize in advance that the main distinction between our approach and that of \cite{MR4169670} lies in the treatment of the low-frequency component,
which captures the essential influence of the positive mass term $\kappa^2\phi$.

To state our main theorems, we first define a Sobolev-type norm:
\begin{Eq}\label{eq:sobolev-norm-associated-with-hkappa}
\|f\|_{H_r^s}:=\kl\|\sqrt{H_\kappa}^{s}f\kr\|_{L_z^r(M^\circ,g)},\quad H^s:=H_2^s,
\end{Eq}
For $f\in C_c^\infty(M^\circ)$, the operator $H_\kappa^{s/2}f$ is defined by the $L^2$ spectral calculus, and $H_r^s$ denotes the completion with respect to the norm above.  When $r\neq2$ and $s<0$, this notation refers to the completion defined above, with $H_\kappa^{s/2}$ first defined on test functions by the spectral calculus.
Since $H_\kappa=H_0+\kappa^2$ and $-\Delta_g=H_0+\frac{(n-1)^2}{4}$, the spectral theorem gives, for fixed $\kappa>0$,
\begin{Eq*}
\|\sqrt{H_\kappa}f\|_{L_z^2(M^\circ,g)}\approx \|\sqrt{-\Delta_g}f\|_{L_z^2(M^\circ,g)}
\approx\|\nabla_g f\|_{L_z^2(M^\circ,g)}.
\end{Eq*}
Moreover, under the spectral assumptions above, $-\Delta_g\geq \frac{(n-1)^2}{4}$ as a quadratic form, and hence $\|\nabla_g f\|_{L_z^2}$ is equivalent to $\|\nabla_g^{\leq1}f\|_{L_z^2}$.  
Here and throughout the paper, whenever $D$ denotes a differential operator or a collection of vector fields, $D^{\leq m}u$ denotes the collection of $D^\alpha u$ with $|\alpha|\leq m$.  
We shall use these equivalent descriptions of the first-order energy norm without further comment.
Next, for $(q,r,\mu)\in[2,\infty]\times(2,\infty]\times \rR$, we define two sets corresponding to admissibility conditions:
\begin{Eq}\label{eq:wave-admissible-triples}
\Lambda_w:=\kl\{(q,r,\mu):\frac{2}{q}\leq(n-1)\kl(\frac{1}{2}-\frac{1}{r}\kr),~2<r\leq\infty,~\mu>s_w\kr\}, 
\end{Eq}
and
\begin{Eq}\label{eq:euclidean-admissible-triples}
\Lambda_e:=&\kl\{(q,r,\mu):\frac{2}{q}\geq(n-1)\kl(\frac{1}{2}-\frac{1}{r}\kr),~2<r\leq\infty,~\mu>s_e\kr\},
\end{Eq}
where $s_w:=n\kl(\frac{1}{2}-\frac{1}{r}\kr)-\frac{1}{q}$ and $s_e:=\frac{n+1}{2}\kl(\frac{1}{2}-\frac{1}{r}\kr)$.
The inequalities $\mu>s_w$ and $\mu>s_e$ are strict; the boundary regularity is not included in the estimates below.

\begin{figure}[H]
\centering
\includegraphics[width=0.94\textwidth]{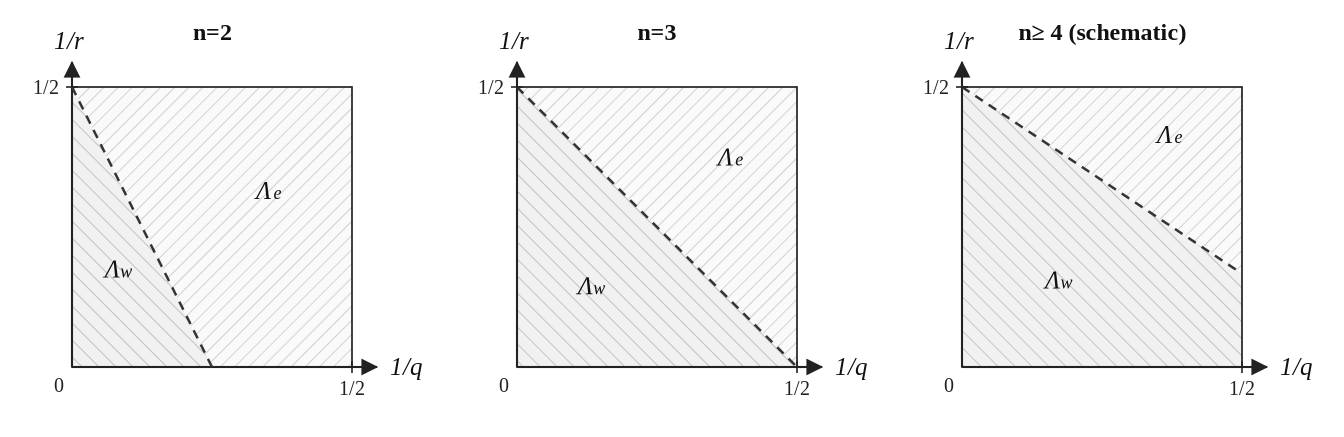}
\caption{The admissible regions in the $(1/q,1/r)$ plane.  The dashed line is $\frac{2}{q}=(n-1)(\frac12-\frac1r)$; the regions below and above it correspond to $\Lambda_w$ and $\Lambda_e$, respectively.  The right panel is schematic for $n\geq4$.}
\label{fig:admissible-regions}
\end{figure}

We are now ready to state our main result concerning the Strichartz estimates.

\begin{theorem}[Strichartz estimates]\label{thm:main-strichartz-estimate}
Let $(M^\circ, g)$ be any non-trapping asymptotically hyperbolic manifold of dimension $n \geq 2$,
and let $\Delta_g$ denote the Laplace-Beltrami operator on $(M^\circ, g)$.
Assume that $-\Delta_g$ has no $L^2$ eigenvalues below $\frac{(n-1)^2}{4}$ and that zero is not a resonance of $H_0$.
Suppose $\phi$ is a solution to the Cauchy problem \eqref{eq:positive-mass-wave-equation} with a fixed $\kappa > 0$.
Then, for every $T\in(0,+\infty]$ and $I=[0,T)$, the following estimate holds for all $(q, r, \mu), (\tilde q, \tr, \tilde \mu) \in \Lambda_w \cup \Lambda_e$:
\begin{Eq}\label{eq:main-strichartz-estimate}
\|\phi\|_{L_t^\infty(I;H^\mu)}+\|\partial_t\phi\|_{L_t^\infty(I;H^{\mu-1})}+\|\phi\|_{L_t^q(I;L_z^r)}
\lesssim \|\psi_0\|_{H^{\mu}}+\|\psi_1\|_{H^{\mu-1}}+\|F\|_{L_t^{\tilde q'}(I;H_{\tr'}^{\mu+\tilde\mu-1})}.
\end{Eq}
\end{theorem}

\begin{remark}
All implicit constants in this paper may depend on the fixed parameter $\kappa>0$; no uniformity as $\kappa\to0^+$ is asserted.
The admissible triples $(q,r,\mu)$ in \cref{thm:main-strichartz-estimate} are the same as those in the shifted wave estimates of \cite{MR4169670}.  Thus the positive mass term $\kappa^2$ does not change the Strichartz range in the $(1/q,1/r)$ plane.  The norms in the estimates, however, are not identical.  In \cite{MR4169670}, the shifted wave estimates are stated in Sobolev spaces with separated high- and low-frequency regularities, for instance $H^{\mu,0}$ and $H^{\mu-1,-\epsilon}$ in their homogeneous estimate and $H_{\tilde r'}^{\mu+\tilde\mu-1,0}$ in the inhomogeneous estimate.  In the present positive-mass problem the low-frequency phase is instead $\sqrt{\lambda^2+\kappa^2}$, which is elliptic near $\lambda=0$; accordingly the natural scale is the single Sobolev scale defined by $H_\kappa=-\Delta_g-\frac{(n-1)^2}{4}+\kappa^2$.  This is the main formal difference between \cref{thm:main-strichartz-estimate} and the shifted wave estimates.
\end{remark}

\subsection{Applications}
Next, we apply the preceding estimates to nonlinear wave equations with small initial data.
We begin with equations involving nonlinearities of the form $\pm|\phi|^p$ or $\pm|\phi|^{p-1}\phi$, namely,
\begin{Eq}\label{eq:strauss-type-nonlinear-equation}
\begin{cases}
\kl(\partial_{t}^2-\Delta_g -\frac{(n-1)^2}{4}+\kappa^2\kr) \phi=F_p(\phi), \quad \kappa>0; \\ (\phi,\partial_t\phi)_{t=0}=\varepsilon(\psi_0,\psi_1),\qquad |F_p(\lambda)|+|\lambda||F_p'(\lambda)|\lesssim |\lambda|^p.
\end{cases}
\end{Eq}
This problem is closely related to the Strauss conjecture in Euclidean space, which concerns the existence of a critical exponent $p_S(n)$ separating global existence from finite-time blow-up for wave equations in $\rR^n$ with similar nonlinearities.

In particular, John \cite{MR526180} showed that in the case $n=3$, the critical exponent is $p_S = 1 + \sqrt{2}$: for $p > p_S$, global existence holds for small initial data, whereas for $1 < p < p_S$, solutions blow up in finite time. Later, Strauss \cite{MR614228} conjectured that in general spatial dimensions $n \geq 2$, the critical exponent $p_S(n)$ is given by the positive root of
\begin{Eq*}
(n-1)p^2-(n+1)p-2=0.
\end{Eq*} 
This conjecture has now been almost completely resolved;
see \cite{chengbo2017recent} and the references therein for a detailed overview of the current state of the field.
For positive-mass Klein-Gordon equations in Euclidean space, the mass term generally improves the small-data theory relative to the wave equation.
For scalar power-type nonlinearities depending only on the unknown, the Hamiltonian structure yields small-data global existence at least for $1<p<1+\frac{4}{n-2}$,
with the upper endpoint understood as $+\infty$ when $n=2$; see, for example, Keel-Tao \cite{MR1738405}.

As in the case of hyperbolic space, we first consider the limiting situation $\kappa = 0$.
For the standard hyperbolic space $\rH^n$,
it was observed in \cite{MR4026182}, building on the work of Georgiev-Lindblad-Sogge \cite{MR1481816} (or Tataru \cite{MR1804518}),
that global existence holds for $p \in (1, p_{conf})$, where $p_{conf}=1+\frac{4}{n-1}$, with small initial data.
(In fact, the Strichartz estimates established in \cite{MR2902129} already suffice to prove this result, although it was not explicitly stated there.)
For perturbed hyperbolic spaces,
the authors of \cite{MR4169670} also proved global existence for $p \in (1, p_{conf})$ using Strichartz estimates.
In addition, in our earlier work \cite{DaiWang2026},
we established global existence for small initial data in the spherically symmetric setting under more general assumptions.
There, we allowed space–time perturbations in the equation's coefficients with only polynomial decay,
and relied on local energy estimates rather than Strichartz-type tools.

In contrast, for the case $\kappa > 0$, 
as mentioned earlier,
the equation in this setting is analogous to the Klein-Gordon equation in Euclidean space.
For the standard hyperbolic space $\rH^n$, small-data global existence for exponents $p>1$ sufficiently close to $1$ was established in
\cite{MR2775816}, \cite{MR3254350} and \cite{MR2944727}.
For perturbed hyperbolic spaces,
the authors of \cite{MR4026182} also proved global existence for small initial data using standard energy methods,
under the assumption $p \in (1, 1 + \frac{2}{n - 2})$, with the upper endpoint understood as $+\infty$ when $n=2$.

In the present paper, we establish global existence and scattering for $p \in (1, 1 + \frac{4}{n-1})$ using Strichartz estimates.
While earlier energy-based arguments yield global existence, the present Strichartz framework also provides the spacetime control needed to obtain scattering.
The emphasis here is on reaching the lower range $p>1$ within the present Strichartz framework; we do not pursue the full expected range of powers.
More precisely, we prove the following:
\begin{theorem}\label{thm:strauss-type-global-existence}
Let $n \geq 2$, and assume that $\kappa > 0$ and $(M^\circ, g)$ are as in \cref{thm:main-strichartz-estimate}.  
Then, for any $p \in (1, 1 + \frac{4}{n-1})$ and $(\psi_0, \psi_1) \in H^{1/2}(M^\circ) \times H^{-1/2}(M^\circ)$,  
there exists a constant $\varepsilon_0 > 0$ such that the Cauchy problem \eqref{eq:strauss-type-nonlinear-equation} admits a unique global solution satisfying
\begin{Eq*}
(\phi,\partial_t\phi)\in L_t^\infty H^{1/2}\times L_t^\infty H^{-1/2},\qquad \phi\in L_{t,z}^{p+1},
\end{Eq*}
whenever $\varepsilon \in (0, \varepsilon_0]$.  

Moreover, let  $L(t)$ denote  the linear wave propagator defined by  
\begin{Eq*}
L(t)(\psi_0,\psi_1):=(\phi_h,\partial_t\phi_h)\quad satisfying\quad L(0)=Id,\quad \kl(\partial_t^2-\Delta_g -\frac{(n-1)^2}{4}+\kappa^2\kr)\phi_h=0.
\end{Eq*}
Then there exists $(\psi_0^+, \psi_1^+) \in H^{1/2} \times H^{-1/2}$ such that 
\begin{Eq}\label{eq:scattering-for-strauss-type-solution}
\lim_{t\rightarrow+\infty}\|(\phi,\partial_t\phi)-L(t)(\psi_0^{+},\psi_1^{+})\|_{H^{1/2}\times H^{-1/2}}=0.
\end{Eq}
\end{theorem}
\begin{remark}
In fact, by adapting the arguments in \cite{MR3254350},
one can further enlarge the admissible range of $p$ under alternative regularity assumptions,
and establish the corresponding scattering results.
\end{remark}

We now turn to equations with nonlinear terms involving powers of derivatives, namely, 
\begin{Eq}\label{eq:glassey-type-derivative-equation}
\begin{cases}
\kl(\partial_{t}^2-\Delta_g -\frac{(n-1)^2}{4}+\kappa^2\kr) \phi=F_p(\partial_t\phi), \quad (t,z)\in \rR_+\times M^\circ; \\ (\phi,\partial_t\phi)_{t=0}=\varepsilon(\psi_0,\psi_1),\qquad |F_p(\lambda)|+|\lambda||F_p'(\lambda)|\lesssim |\lambda|^p.
\end{cases}
\end{Eq}
This type of equation is related to the so-called Glassey conjecture in Euclidean space,
which postulates the existence of a critical exponent $p_G(n)$ that separates global existence from finite-time blow-up
for wave equations in $\rR^n$ with such derivative-type nonlinearities.

More specifically, in his review of Sideris' work \cite{MR711440}, Glassey conjectured that the critical exponent should be $p_G = 1 + \frac{2}{n-1}$. 
A proof of the blow-up part of this conjecture was given by Zhou \cite{MR1845748}. 
Regarding global existence, the cases $n=2,3$ were treated independently by Hidano-Tsutaya \cite{MR1386769} and Tzvetkov \cite{MR1637692}. 
In higher dimensions, the sharp global-existence problem for general nonradial data remains open; radial-data results include, for example, Hidano-Wang-Yokoyama \cite{MR2980460}.
For general, not necessarily Hamiltonian, Klein-Gordon equations with nonlinearities of the form $F(u,u_t,\nabla u)$, energy heuristics suggest the threshold $p>1+\frac{2}{n}$ in the Euclidean setting; see Lindblad-Sogge \cite{LdSo96-2}, Keel-Tao \cite{MR1738405}, and the references therein.  The corresponding sharp theory for nonlinearities involving only derivatives has not yet been established in general.

For the case of hyperbolic space in the limiting situation $\kappa = 0$,
a related result in the hyperbolic limiting case was obtained in our previous work \cite{DaiWang2026}.
There, using local energy estimates,
we proved that in the spherically symmetric setting,
global existence for small initial data holds for all $p \in (1, 1 + \frac{2}{n - 2})$,
even when the metric includes certain perturbations.

In contrast, for the case $\kappa > 0$,
it was also observed in \cite{DaiWang2026} that the corresponding local energy estimates give rise to a linear error term,
which prevents a direct application of the earlier argument.
In the present paper,
we overcome this difficulty within a unified framework by making use of \cref{thm:main-strichartz-estimate}.
The goal here is not to identify a Glassey-type critical exponent, but to close the low-power derivative nonlinear problem in the positive-mass hyperbolic framework; the upper bound on $p$ below comes from the Hardy and local energy weights used in the proof.
Accordingly, we now state the assumptions on the metric $g$ in an alternative coordinate system, following the approach of \cite{DaiWang2026}.
The point of the additional structure is to connect two different tools: the Strichartz estimates of \cref{thm:main-strichartz-estimate}, which are formulated on $(M^\circ,g)$, and the Hardy-type inequalities and local energy estimates in the hyperbolic coordinate framework associated with $\eta$.
Thus, in the derivative nonlinear applications below, the metric is required to belong simultaneously to the asymptotically hyperbolic Strichartz framework and to a perturbative hyperbolic framework.
The Hardy-type estimate is recalled at the point where it is used, while the positive-mass local energy estimate is proved in Appendix~\ref{sec:appendix}.

The following assumption is used only for the derivative nonlinear application and is independent of the linear Strichartz theorem.
\begin{hypothesis}\label{hyp:small-asymptotic-decay}
It is well known that the metric $\eta$ of the standard hyperbolic space $\rH^n$ can be written in polar coordinates $(r, \Theta) \in \rR_+ \times \mathbb{S}^{n-1}$ as
\begin{Eq*}
\eta=\d r^2+\sinh^2 r\d\Theta^2.
\end{Eq*}

Let $I\subset\rR_+$ be an interval.  In the region $\{r\in I\}$, we assume that $g$ can be written in the same $(r,\Theta)$ coordinate system and compared with the standard hyperbolic metric $\eta$ there; the expression $g-\eta$ below is understood in this coordinate representation.
We say that $g$ satisfies \cref{hyp:small-asymptotic-decay} on $I$ if the metric perturbation $\gamma := g - \eta$ satisfies the following decay condition on $\{r\in I\}$:
for any $m_1, m_2 \in \mathbb{N}_0$, there exists a constant $C_{m_1, m_2}$ such that  
\begin{Eq*}
\kl|\varL_\Omega^{m_1}\nabla_\eta^{m_2} \gamma\kr|_\eta\leq \varepsilon'C_{m_1,m_2} e^{-br},\qquad |T|_\eta^2:=\kl<T,T\kr>_\eta,
\end{Eq*} 
with $b\in(0,1)$ and $\varepsilon' > 0$ small enough.
Here, $\nabla_\eta$ denotes the covariant derivative with respect to $\eta$, and $\mathcal{L}_\Omega$ denotes the Lie derivative along rotational vector fields.
\end{hypothesis}
\begin{remark}
Under \cref{hyp:small-asymptotic-decay} on the coordinate region under consideration, after taking $\varepsilon'$ sufficiently small, the metrics $g$ and $\eta$ are uniformly comparable in the $(r,\Theta)$ coordinate system.  In particular,
\begin{Eq*}
dV_g\approx dV_\eta,\qquad |T|_g\approx |T|_\eta,
\end{Eq*}
hold for any tensor $T$ in that region.  In particular, the corresponding spatial $L^p$ norms are equivalent for the exponents used below whenever this comparison is used.  This allows us to combine the $g$-based Strichartz estimates from \cref{thm:main-strichartz-estimate} with the $\eta$-based Hardy and local energy estimates used in the derivative nonlinear arguments.
\end{remark}

We now state the following global existence result together with the corresponding scattering property.
\begin{theorem}\label{thm:glassey-type-global-existence}
Let $n \geq 3$, and assume that $\kappa > 0$ and $(M^\circ, g)$ are as in \cref{thm:main-strichartz-estimate}, with $g$ spherically symmetric.  
The spherical symmetry is understood with respect to the same $(r,\Theta)$ coordinate system.
Furthermore, assume that $g$ satisfies \cref{hyp:small-asymptotic-decay} with $I=\rR_+$.
Then, for any $p \in (1, 1 + \frac{1}{n - 2})$ and spherically symmetric data $(\psi_0, \psi_1) \in H^2(M^\circ) \times H^1(M^\circ)$,  
there exists a constant $\varepsilon_0 > 0$ such that the Cauchy problem \eqref{eq:glassey-type-derivative-equation} admits a unique global solution $(\phi,\partial_t\phi)\in L_t^\infty H^2\times L_t^\infty H^1$ whenever $\varepsilon \in (0, \varepsilon_0]$.  

Moreover, let $L(t)$ denote  the linear wave propagator defined in \cref{thm:strauss-type-global-existence}.  
Then, there exists $(\psi_0^+, \psi_1^+) \in H^1 \times L^2$ such that  
\begin{Eq}\label{eq:scattering-for-glassey-type-solution}
\lim_{t\rightarrow+\infty}\|(\phi,\partial_t\phi)-L(t)(\psi_0^{+},\psi_1^{+})\|_{H^{1}\times L^2}=0.
\end{Eq}
\end{theorem}
\begin{remark}\label{rem:spherically-symmetric-asymptotically-hyperbolic-metrics}
For the standard hyperbolic metric $\frac{\d x^2 + (1 - x^2)^2 \d \Theta^2 / 4}{x^2}$,  
we may perform the change of variable $x = e^{-r}$, which transforms the metric into $\d r^2 + \sinh^2 r \, \d \Theta^2$.  

For spherically symmetric asymptotically hyperbolic metrics of the form $g = \frac{\d x^2 + h(x)\d \Theta^2}{x^2}$,  
we may apply the change of variables  $x=2\sqrt{h(0)}e^{-r}$ in $x\ll 1$, yielding
\begin{Eq*}
g=\d r^2+\frac{e^{2r}h(2\sqrt{h(0)}e^{-r})}{4h(0)}\d\Theta^2=\d r^2+\frac{e^{2r}}{4}(1+O(e^{-r}))\d\Theta^2=\eta +O(e^{-r})\sinh^2r\d\Theta^2.
\end{Eq*}
Here $O(e^{-r})$ denotes the relative error in the angular coefficient; equivalently, the perturbation has size $O(e^{-r})$ in the $\eta$-tensor norm.
Hence, the condition \cref{hyp:small-asymptotic-decay} holds for large $r$.
In other words, if $g$ is sufficiently close to $\eta$ also in the interior region, 
it naturally satisfies \cref{hyp:small-asymptotic-decay} with $I=\rR_+$.
\end{remark}

Based on the two results above, it is natural to expect that,
at least within the unified framework,
global existence for small initial data should also hold in dimensions $n \geq 3$
for more general smooth nonlinearities depending on the unknown, its time derivative, and its spatial exterior derivative.

For the wave equation in the Euclidean setting, consider standard quadratic nonlinearities in first derivatives, say $F(\partial\phi)$ with $F(0)=F'(0)=0$. It is well known that, for small initial data, global existence for this derivative-type problem holds when $n \geq 4$, whereas in three dimensions ($n=3$) an additional null condition is required (see Klainerman \cite{MR837683}).

In the hyperbolic space with the limiting case $\kappa = 0$,
we showed in our previous work \cite{DaiWang2026} that for $n \geq 3$,
global existence for small initial data holds without requiring any null condition, even for a class of quasilinear shifted-wave equations.

As for the case $\kappa > 0$, 
a related geometric result is due to Lawrie-Oh-Shahshahani \cite{MR3801478},
who proved global existence for small initial data in the wave map system on hyperbolic spaces of dimension $n \geq 4$.
Their formulation involves a geometric null structure, and should not be viewed as a direct special case of the scalar positive-mass model considered here.
Under the local energy hypothesis below, together with the exterior perturbative assumption used in \cref{hyp:small-asymptotic-decay}, \cref{thm:smooth-nonlinearity-global-existence} gives a corresponding semilinear small-data result without imposing any null condition.

Finally, in analogy with the discussion on asymptotically Euclidean settings,
we broaden our analytical framework.
We no longer require \cref{hyp:small-asymptotic-decay} with $I=\rR_+$.
Instead, in the interior region we assume a local energy hypothesis of the following form.

\begin{hypothesis}\label{hyp:local-energy}
Assume that the metric $g$ can be expressed in the $(r, \Theta)$ coordinate system,
and set
\begin{Eq*}
P_g:=\partial_t^2-\Delta_g-\frac{(n-1)^2}{4}+\kappa^2.
\end{Eq*}
For some fixed $0 < a \ll 1$ and any $R \in \rR_+$, we assume that
\begin{Eq*}
\|(\partial_t,\nabla_g^{\leq 1})\phi\|_{L_{t,z}^2(r<R)}\lesssim \|\psi_0\|_{H^1}+\|\psi_1\|_{L^2}+\|F\|_{L_{t,z}^2}+\|e^{-ar}\phi\|_{L_{t,z}^2},
\end{Eq*}
for any solution of $P_g\phi=F$ with initial data $(\psi_0, \psi_1)$ and forcing term $F(t,z)$ supported in $\{r \leq R \}$. 
The implicit constant here may depend on $a$ and $R$.
\end{hypothesis}
\begin{remark}
In the applications below, after fixing the exterior decay rate $b$ in \cref{hyp:small-asymptotic-decay}, we decrease $a$ if necessary so that $0<2a<b$.
\end{remark}
\begin{remark}
We expect that, under suitable non-trapping and spectral assumptions, the local energy estimate above should hold without the last term
$\|e^{-ar}\phi\|_{L_{t,z}^2}$.
We keep this term in \cref{hyp:local-energy} because it appears naturally in the uniformly local energy estimate for small perturbations proved in \cref{prop:positive-mass-local-energy-estimate}, and because it does not affect the subsequent nonlinear argument.

For compactly supported sufficiently small perturbations of the hyperbolic metric, the multiplier method in our earlier work \cite{DaiWang2026} should yield the same local estimate without this zeroth-order remainder on each fixed compact region.  Indeed, for a prescribed radius $R$ one may choose a bounded radial multiplier of the form
\begin{Eq*}
h(r)=1+\int_0^r \exp(-M\rho^2)\d\rho
\end{Eq*}
with $M$ sufficiently large, depending on $R$ and the compact support of the perturbation.  Since $h'=\exp(-Mr^2)$ is bounded from below on $[0,R]$, it controls the local spacetime norm there.  Moreover, in the zeroth-order part of the Morawetz identity, the choice $\tilde w=w-r^{-1}$ gives a positive contribution containing $-h''/r=2M\exp(-Mr^2)$, which absorbs the unfavorable mass term and the small compactly supported perturbative errors.  For more general asymptotically hyperbolic metrics, however, such a remainder-free estimate would require a separate resolvent or local energy argument, and we do not use it in the present paper.
\end{remark}

We now present the final theorem of this paper.
\begin{theorem}\label{thm:smooth-nonlinearity-global-existence}
Let $n \geq 3$, and assume that $\kappa > 0$ and that $(M^\circ, g)$ is as in \cref{thm:main-strichartz-estimate}.
Furthermore, suppose that the metric $g$ can be expressed in the $(r, \Theta)$ coordinate system, satisfies \cref{hyp:local-energy}, and satisfies \cref{hyp:small-asymptotic-decay} with $I=[R_0,+\infty)$ for some sufficiently large $R_0$.

Consider the equation
\begin{Eq}\label{eq:smooth-nonlinearity-equation}
\begin{cases}
\kl(\partial_{t}^2-\Delta_g -\frac{(n-1)^2}{4}+\kappa^2\kr) \phi=F_2(\phi,\partial_t\phi)+F_0(\phi,\partial_t\phi)(\d\phi,\d\phi)_g,\\
(\phi,\partial_t\phi)_{t=0}=\varepsilon(\psi_0,\psi_1).
\end{cases}
\end{Eq}
where $F_0,F_2\in C^\infty(\rR^2)$, $F_2(0,0)=DF_2(0,0)=0$, and $(\psi_0,\psi_1)\in C_0^\infty(M^\circ)\times C_0^\infty(M^\circ)$.
Then, there exists a constant $\varepsilon_0 > 0$ such that, for any $\varepsilon \in (0, \varepsilon_0]$, the equation admits a unique global smooth solution.
\end{theorem}

\begin{remark}
Here $\d\phi$ denotes the spatial exterior derivative of $\phi$, so it is a spatial covector.
The argument can also apply to more general smooth nonlinearities of the form $F(\phi,\partial_t\phi,\d\phi)$,
provided that a precise convention is fixed for the covector variable, 
for instance through a fixed background trivialization, 
or more intrinsically as a fiberwise function on $T^*M^\circ$. To avoid these additional notational complications, we restrict ourselves here to the representative special form $F_0(\phi,\partial_t\phi)(\d\phi,\d\phi)_g$.
We emphasize that the nonlinear term in \eqref{eq:smooth-nonlinearity-equation} is not a null form, and no null condition is imposed.
\end{remark}

The paper is organized as follows.
In \cref{sec:preparation}, we provide preliminary material, including notations and properties of the spectral measure associated with $-\Delta_g$.
In \cref{sec:low-frequency-half-wave-strichartz} and \cref{sec:high-frequency-half-wave-strichartz}, we establish Strichartz estimates for the low-frequency and localized high-frequency components of the half-wave propagator, respectively.
Then, in \cref{sec:strichartz-estimates-for-main-equation}, we combine these two results to derive the full Strichartz estimate for the linear equation \eqref{eq:positive-mass-wave-equation}, stated in \cref{thm:main-strichartz-estimate}.
Finally, in \cref{sec:applications}, we present several applications of \cref{thm:main-strichartz-estimate} to different classes of nonlinear wave equations, stated in \cref{thm:strauss-type-global-existence}, \cref{thm:glassey-type-global-existence} and \cref{thm:smooth-nonlinearity-global-existence}.
The Appendix contains the proof of the local energy estimates used in the derivative nonlinear applications.

\section{Preparation}\label{sec:preparation}

First, we introduce a dyadic partition of unity:
$1=\chi(s)+\sum_{j\geq 1}\beta_j(s)$ for $s\geq 0$, where
$\chi$ is smooth on $[0,+\infty)$, equals $1$ on $[0,1]$, and is supported in $[0,2]$,
and $\beta_j(s):=\chi(2^{-j}s)-\chi(2^{-j+1}s)$ is supported in $[2^{j-1},2^{j+1}]$.

We next recall the spectral measure estimates for $H_0$ that will be used below; see, for example, \cite{MR3473907,MR3805767}.
\begin{proposition}
Under the geometric and spectral assumptions of \cref{thm:main-strichartz-estimate}, the following low-frequency and microlocal high-frequency spectral measure estimates hold.
For the low frequency, i.e. $\lambda\leq 2$, the Schwartz kernel of the spectral measure $E_{\sqrt{H_0}}'(\lambda;z,z')$ satisfies
\begin{Eq}\label{eq:spectral-measure-low-frequency-expansion}
E_{\sqrt{H_0}}'(\lambda;z,z')=\lambda\kl((\rho_L\rho_R)^{\frac{n-1}{2}+\II \lambda}a(\lambda;z,z')-(\rho_L\rho_R)^{\frac{n-1}{2}-\II \lambda}a(-\lambda;z,z')\kr),
\end{Eq}
where $a\in C^\infty([-2,2]\times M_0^2)$ and $\rho_L$ and $\rho_R$ are respectively the boundary defining functions for the left and right boundary in the double space $M_0^2$. 
Furthermore, there holds
\begin{Eq}\label{eq:spectral-measure-low-frequency-estimate}
|E_{\sqrt{H_0}}'(\lambda;z,z')|\lesssim \lambda^2\big(1+d(z,z')\big)e^{-(n-1)d(z,z')/2}.
\end{Eq}

For the high frequency, i.e. $\lambda\geq 1/2$, 
there exists a finite pseudo-differential operator partition of the identity operator
\begin{Eq}\label{eq:microlocal-partition-of-identity}
\mathrm{Id}=\sum_{k=1}^KQ_k(\lambda),
\end{Eq}
where the $Q_k$ are uniformly bounded as operators on $L^2$ and $K$ is independent of $\lambda$, such that
\begin{Eq}\label{eq:microlocalized-spectral-measure-expansion}
\kl(Q_k(\lambda)E_{\sqrt{H_0}}'(\lambda)Q_k^*(\lambda)\kr)(z,z')=\lambda^{n-1}\sum_\pm e^{\pm\II\lambda d(z,z')}a_{\pm}(\lambda;z,z')+b(\lambda;z,z'). 
\end{Eq}
Meanwhile, for any $\alpha\in\rN_0$, the amplitudes $a_{\pm}$ satisfy
\begin{Eq}\label{eq:oscillatory-amplitude-estimate}
\kl|\partial_\lambda^\alpha a_{\pm}(\lambda;z,z')\kr|\lesssim\begin{cases}
\lambda^{-\alpha}\big(1+\lambda d(z,z')\big)^{-\frac{n-1}{2}},&d(z,z')\leq 1,\\
\lambda^{-\frac{n-1}{2}-\alpha}e^{-(n-1)d(z,z')/2},&d(z,z')\geq 1,
\end{cases}
\end{Eq}
with the implicit constant depending only on $\alpha$, and $b$ satisfies
\begin{Eq}\label{eq:remainder-amplitude-estimate}
\kl|\partial_\lambda^\alpha b(\lambda;z,z')\kr|\lesssim
\lambda^{-M-\alpha}e^{-(n-1)d(z,z')/2},
\end{Eq}
for any $M\geq 0$ and the implicit constant depending only on $\alpha$ and $M$.

Moreover, if $(M^\circ,g)$ is additionally simply connected with nonpositive sectional curvature, 
then the same estimates hold for the spectral measure without microlocalization. 
That is, in this case we can take $\{Q_k(\lambda)\}$ to be the trivial partition of unity.
\end{proposition}

In order to study the equation \eqref{eq:positive-mass-wave-equation}, we are led to focus on the operator $H_\kappa = H_0 + \kappa^2$.  
For any $\sigma \in \rR$, we define the weighted half-wave propagator by  
\begin{Eq*}
U^\sigma(t):=&e^{\II t\sqrt{H_\kappa}}H_\kappa^{-\sigma/2}=\int_0^\infty e^{\II t\sqrt{\lambda^2+\kappa^2}}\sqrt{\lambda^2+\kappa^2}^{-\sigma}E_{\sqrt{H_0}}'(\lambda)\d\lambda,
\end{Eq*}
We further decompose this operator into its low-frequency and localized high-frequency parts:  
\begin{Eq*}
U^{\sigma,low}(t):=&\int_0^\infty e^{\II t\sqrt{\lambda^2+\kappa^2}}\chi(\lambda)\sqrt{\lambda^2+\kappa^2}^{-\sigma}E_{\sqrt{H_0}}'(\lambda)\d\lambda,\\
U^{\sigma,high}(t):=&U^\sigma(t)-U^{\sigma,low}(t)=\sum_{k=1}^K\sum_{j=1}^\infty U_{j,k}^{\sigma,high}(t),\\
U_{j,k}^{\sigma,high}(t):=&\int_0^\infty e^{\II t\sqrt{\lambda^2+\kappa^2}}\beta_j(\lambda)\sqrt{\lambda^2+\kappa^2}^{-\sigma}Q_k(\lambda)E_{\sqrt{H_0}}'(\lambda)\d\lambda,
\end{Eq*}
where the operators $Q_k$ are defined in \eqref{eq:microlocal-partition-of-identity}.
In the $TT^*$ argument below, this microlocalized definition is used in the standard sense of the spectral measure decomposition above, so that the kernel of $U_{j,k}^{\sigma,high}(t)[U_{j,k}^{\sigma,high}]^*(\tau)$ is expressed through $[Q_k(\lambda)E_{\sqrt{H_0}}'(\lambda)Q_k^*(\lambda)]$.

Finally, we recall the mixed-exponent Kunze-Stein type inequality in the kernel-dominated form used below.  The hyperbolic-space version follows from \cite[Lemma~5.1]{MR2902129}; for the localization of exponentially decaying kernels to the asymptotically hyperbolic setting, see the proof of \cite[Lemma~2.1]{MR4169670}.
\begin{proposition}\label{prop:kunze-stein-convolution-inequality}
Let $T$ be an integral operator whose kernel is dominated by a radial function $K(d(z,z'))$.  For every $2<r,\tr<\infty$, there holds
\begin{Eq*}
\|Tf\|_{L_z^r}\lesssim \|f\|_{L_z^{\tr'}}\kl(\int_0^\infty  (1+R)^{\alpha}e^{(n-1)\kl(1-\frac{\alpha}{2}\kr)R}|K(R)|^{\theta}\d R\kr)^{\frac{1}{\theta}},
\end{Eq*}
where $\alpha=\frac{2\min\{r,\tr\}}{r+\tr}$ and $\theta=\frac{r\tr}{r+\tr}$.
\end{proposition}

\section{Strichartz estimates for the low-frequency half-wave propagator}\label{sec:low-frequency-half-wave-strichartz}

In this section, 
we establish that the low-frequency component $U^{\sigma,low}(t)$ of the half-wave propagator satisfies the following estimate.
\begin{proposition}\label{prop:low-frequency-half-wave-strichartz}
Let $U^{\sigma,low}(t)$ be defined as in \cref{sec:preparation}. 
Then, for any pair $(q, r) \in [2, \infty] \times (2, \infty]$ and any $\sigma \in \rR$, we have  
\begin{Eq*}
\kl\|U^{\sigma,low}(t)f\kr\|_{L_t^qL_z^r}\lesssim \|f\|_{L^2}.
\end{Eq*}
The implicit constant may depend on the fixed parameter $\kappa>0$, on $\sigma$, and on the geometry of $(M^\circ,g)$.
\end{proposition}

\subsection{Proof of Proposition~\ref{prop:low-frequency-half-wave-strichartz}}
Following the $TT^*$ method, 
to prove this proposition, 
it suffices to show that for any $(q, r)$ as above and any functions $F, G$, the following estimate holds:  
\begin{Eq}\label{eq:low-frequency-half-wave-bilinear-estimate}
\kl|\iint\kl<[U^{\sigma,low}]^*(\tau)F(\tau),[U^{\sigma,low}]^*(t)G(t)\kr>\d\tau\d t\kr|
\lesssim\|F(\tau)\|_{L_\tau^{q'}L_z^{r'}}\|G(t)\|_{L_t^{q'}L_z^{r'}}.
\end{Eq}

This estimate, in turn, is a consequence of the following lemma.
\begin{lemma}\label{lem:low-frequency-half-wave-kernel-estimate}
For the operator $U^{\sigma,low}(t)[U^{\sigma,low}]^*(\tau)$ and any fixed $0 < \delta \ll 1$, we have the following kernel estimate:  
\begin{Eq}\label{eq:low-frequency-half-wave-kernel-estimate}
I(t-\tau):=&\ker \kl(U^{\sigma,low}(t)[U^{\sigma,low}]^*(\tau)\kr)\\
=&\int_0^\infty e^{\II (t-\tau)\sqrt{\lambda^2+\kappa^2}}\chi^2(\lambda)\sqrt{\lambda^2+\kappa^2}^{-2\sigma}E_{\sqrt{H_0}}'(\lambda)\d\lambda,\\
\kl|I(t)\kr|\lesssim& (1+|t|)^{-\frac{3}{2}}e^{-(\frac{n-1}{2}-\delta)d(z,z')}.
\end{Eq}
\end{lemma}

Assuming this lemma holds, we first consider the case $r\in(2,\infty)$.  Fixing $\delta\ll1$ and applying H\"older's inequality together with \cref{prop:kunze-stein-convolution-inequality}, we derive
\begin{Eq*}
&\kl<[U^{\sigma,low}]^*(\tau)F(\tau),[U^{\sigma,low}]^*(t)G(t)\kr>\\
\lesssim&\kl\|U^{\sigma,low}(t)[U^{\sigma,low}]^*(\tau)F(\tau)\kr\|_{L_z^r}\|G(t)\|_{L_z^{r'}}\\
\lesssim&(1+|t-\tau|)^{-\frac{3}{2}}\kl(\int_0^\infty (1+R)e^{(n-1)\kl(1-\frac{1}{2}-\frac{r}{4}\kr)R+\frac{r\delta}{2}R}\d R\kr)^{\frac{2}{r}}\|F(\tau)\|_{L_z^{r'}}\|G(t)\|_{L_z^{r'}}\\
\lesssim&(1+|t-\tau|)^{-\frac{3}{2}}\|F(\tau)\|_{L_z^{r'}}\|G(t)\|_{L_z^{r'}}.
\end{Eq*}
When $r=\infty$, the same spatial estimate follows directly from the pointwise kernel bound in \cref{lem:low-frequency-half-wave-kernel-estimate}:
\begin{Eq*}
\kl\|U^{\sigma,low}(t)[U^{\sigma,low}]^*(\tau)F(\tau)\kr\|_{L_z^\infty}
\lesssim(1+|t-\tau|)^{-\frac{3}{2}}\|F(\tau)\|_{L_z^1}.
\end{Eq*}
Furthermore, for any $q\in[2,\infty]$, we obtain
\begin{Eq*}
&\kl|\iint\kl<[U^{\sigma,low}]^*(\tau)F(\tau),[U^{\sigma,low}]^*(t)G(t)\kr>\d\tau\d t\kr|\\
\lesssim&\kl\|(1+|t|)^{-\frac{3}{2}}\kr\|_{L_t^{q/2}}\|F(\tau)\|_{L_\tau^{q'}L_z^{r'}}\|G(t)\|_{L_t^{q'}L_z^{r'}}\\
\lesssim&\|F(\tau)\|_{L_\tau^{q'}L_z^{r'}}\|G(t)\|_{L_t^{q'}L_z^{r'}},
\end{Eq*}
with the evident interpretation of the time norm when $q=\infty$.
This yields precisely \eqref{eq:low-frequency-half-wave-bilinear-estimate}. 
Consequently, to establish \cref{prop:low-frequency-half-wave-strichartz}, it remains only to verify \cref{lem:low-frequency-half-wave-kernel-estimate}.

\subsection{Proof of Lemma~\ref{lem:low-frequency-half-wave-kernel-estimate}}
By \eqref{eq:spectral-measure-low-frequency-estimate}, for any given $\delta$, we immediately derive
\begin{Eq*}
\kl|I(t)\kr|\lesssim&\big(1+d(z,z')\big)e^{-(n-1)d(z,z')/2}\int_0^{2} \lambda^2\sqrt{\lambda^2+\kappa^2}^{-2\sigma}\d\lambda\lesssim e^{-(\frac{n-1}{2}-\delta)d(z,z')}.
\end{Eq*}
This establishes the portion of \eqref{eq:low-frequency-half-wave-kernel-estimate} corresponding to the case $|t|\lesssim 1$.

Next, we consider the case $|t| \gtrsim 1$ and decompose $I(t)$ as $I(t) := I_0(t) + I_\infty(t)$, where
\begin{Eq*}
I_0(t):=&\int_0^\infty e^{\II t\sqrt{\lambda^2+\kappa^2}}\chi(|t|^{1/2}\lambda)\chi^2(\lambda)\sqrt{\lambda^2+\kappa^2}^{-2\sigma}E_{\sqrt{H_0}}'(\lambda)\d\lambda,\\
I_\infty(t):=&I(t)-I_0(t)=\int_0^\infty e^{\II t\sqrt{\lambda^2+\kappa^2}}[1-\chi](|t|^{1/2}\lambda)\chi^2(\lambda)\sqrt{\lambda^2+\kappa^2}^{-2\sigma}E_{\sqrt{H_0}}'(\lambda)\d\lambda.
\end{Eq*}
Here, for $I_0(t)$, applying \eqref{eq:spectral-measure-low-frequency-estimate} with $|t| > 2$ yields
\begin{Eq*}
\kl|I_0(t)\kr|\lesssim&\big(1+d(z,z')\big)e^{-(n-1)d(z,z')/2}\int_0^{2|t|^{-1/2}} \sqrt{\lambda^2+\kappa^2}^{-2\sigma}\lambda^2\d\lambda\lesssim |t|^{-\frac{3}{2}}e^{-(\frac{n-1}{2}-\delta)d(z,z')}.
\end{Eq*}

As for $I_\infty(t)$, a more refined analysis is required. 
We begin by stating the following technical claim:
\begin{claim}\label{claim:low-frequency-oscillatory-integral-expansion}
For any fixed $M\in\rN_0$ and $|t|\geq 2$, the term $I_\infty(t)$ can be rewritten as:
\begin{Eq*}
I_\infty(t)=(\rho_L\rho_R)^{\frac{n-1}{2}}\sum_{\bf{M}}C_{\bf{M}}\int_{|t|^{-1/2}/2}^{2}&|t|^{\frac{M_3}{2}-M}\Big(e^{\II t\sqrt{\lambda^2+\kappa^2}}\lambda^{1-M-M_1+M_2}\sqrt{\lambda^2+\kappa^2}^{-2M_2+M-2\sigma}\\
&\cdot[1-\chi]^{(M_3)}(|t|^{1/2}\lambda)[\chi^2]^{(M_4)}(\lambda)\\
&\cdot\partial_\lambda^{M_5}\kl((\rho_L\rho_R)^{\II \lambda}a(\lambda;z,z')-(\rho_L\rho_R)^{-\II \lambda}a(-\lambda;z,z')\kr)\Big)\d \lambda,
\end{Eq*}
where ${\bf M}=(M_1,M_2,M_3,M_4,M_5)$, whose components are all non-negative and satisfy $M=\sum_{i=1}^5 M_i$.
\end{claim}

Using this claim and the smoothness property $a\in C^\infty([-2,2]\times M_0^2)$, 
 we derive the key estimate
\begin{Eq*}
\kl|\partial_\lambda^{M_5}\kl((\rho_L\rho_R)^{\II \lambda}a(\lambda;z,z')-(\rho_L\rho_R)^{-\II \lambda}a(-\lambda;z,z')\kr)\kr|\lesssim 
\lambda^{\max\{1-M_5,0\}}\cdot (1+|\ln(\rho_L\rho_R)|)^{\max\{M_5,1\}}.
\end{Eq*}
The factor inside the derivative vanishes at $\lambda=0$; this gives one power of $\lambda$ when $M_5=0$, while higher $\lambda$-derivatives only produce logarithmic losses from the boundary factors.
We shall also use the standard relation between the boundary defining functions on the double space and the physical distance,
\begin{Eq*}
(\rho_L\rho_R)^{\frac{n-1}{2}}\big(1+|\ln(\rho_L\rho_R)|\big)^M
\lesssim (1+d(z,z'))^M e^{-\frac{n-1}{2}d(z,z')}.
\end{Eq*}
Taking $M \geq 4$ and applying \cite[Proposition 3.4]{MR3805767}, we obtain
\begin{Eq*}
|I_\infty(t)|\lesssim &(\rho_L\rho_R)^{\frac{n-1}{2}}\sum_{\bf{M}}(1+|\ln(\rho_L\rho_R)|)^{\max\{M_5,1\}}|t|^{\frac{M_3}{2}-M}\int_{|t|^{-1/2}/2}^{2}\lambda^{1-M-M_1+M_2+\max\{1-M_5,0\}} \d \lambda\\
\lesssim &(1+|\ln(\rho_L\rho_R)|)^{M}(\rho_L\rho_R)^{\frac{n-1}{2}}\sum_{\bf{M}}|t|^{\frac{M_3}{2}-M}\int_{|t|^{-1/2}/2}^{2}\lambda^{2-M-M_1-M_5} \d \lambda\\
\lesssim &(1+d(z,z'))^{M}e^{-\frac{n-1}{2}d(z,z')}\sum_{\bf{M}}|t|^{\frac{M_1+M_3+M_5-M-3}{2}}\\
\lesssim& e^{-(\frac{n-1}{2}-\delta)d(z,z')}|t|^{-\frac{3}{2}}.
\end{Eq*}
Combining the estimates above completes the proof of \cref{lem:low-frequency-half-wave-kernel-estimate}, modulo the verification of \cref{claim:low-frequency-oscillatory-integral-expansion}.

\subsection{Proof of Claim~\ref{claim:low-frequency-oscillatory-integral-expansion}}
We prove \cref{claim:low-frequency-oscillatory-integral-expansion} by induction on $M$.
The case $M=0$ is immediate.
Assume that the claim holds for some $M\geq0$.
We use the identity
\begin{Eq*}
e^{\II t\sqrt{\lambda^2+\kappa^2}}=\kl[\frac{\sqrt{\lambda^2+\kappa^2}}{\II t\lambda}\partial_\lambda\kr] \kl(e^{\II t\sqrt{\lambda^2+\kappa^2}}\kr),
\end{Eq*}
All boundary terms vanish in the integrations by parts below, because the cutoffs make the integrand compactly supported in $\lambda$ away from the endpoints of the integration interval.
Therefore, we have
\begin{Eq*}
&\int_{|t|^{-1/2}/2}^{2}|t|^{\frac{M_3}{2}-M}e^{\II t\sqrt{\lambda^2+\kappa^2}}\lambda^{1-M-M_1+M_2}\sqrt{\lambda^2+\kappa^2}^{-2M_2+M-2\sigma}\\
&\phantom{\int_{|t|^{-1/2}/2}^{2}}\quad\cdot[1-\chi]^{(M_3)}(|t|^{1/2}\lambda)[\chi^2]^{(M_4)}(\lambda)\partial_\lambda^{M_5}\kl((\rho_L\rho_R)^{\II \lambda}a(\lambda;z,z')-(\rho_L\rho_R)^{-\II \lambda}a(-\lambda;z,z')\kr)\d \lambda\\
=&-\II\int_{|t|^{-1/2}/2}^{2}|t|^{\frac{M_3}{2}-(M+1)}\partial_\lambda\kl(e^{\II t\sqrt{\lambda^2+\kappa^2}}\kr)\lambda^{1-(M+1)-M_1+M_2}\sqrt{\lambda^2+\kappa^2}^{-2M_2+(M+1)-2\sigma}\\
&\phantom{\int_{|t|^{-1/2}/2}^{2}}\quad\cdot[1-\chi]^{(M_3)}(|t|^{1/2}\lambda)[\chi^2]^{(M_4)}(\lambda)\partial_\lambda^{M_5}\kl((\rho_L\rho_R)^{\II \lambda}a(\lambda;z,z')-(\rho_L\rho_R)^{-\II \lambda}a(-\lambda;z,z')\kr)\d \lambda.
\end{Eq*}

Next, we perform integration by parts. 
When $\partial_\lambda$ falls on $\lambda^{a}$, 
we get $a \lambda^{a-1}$, 
which corresponds to $M_1 \to M_1 + 1$;
when $\partial_\lambda$ falls on $\sqrt{\lambda^2+\kappa^2}^{a}$, 
we get $a \lambda \sqrt{\lambda^2+\kappa^2}^{a-2}$, 
which corresponds to $M_2 \to M_2 + 1$;
when $\partial_\lambda$ falls on $[1-\chi]^{(M_3)}(|t|^{1/2}\lambda)$, 
we get $|t|^{\frac{1}{2}} [1-\chi]^{(M_3+1)}(|t|^{1/2}\lambda)$, 
which corresponds to $M_3 \to M_3 + 1$;
finally, when $\partial_\lambda$ falls on the last two terms, 
they correspond to $M_4 \to M_4 + 1$ and $M_5 \to M_5 + 1$ respectively.
This completes the induction.

\section{Strichartz estimates for the localized high-frequency half-wave propagator}\label{sec:high-frequency-half-wave-strichartz}
In this section, we establish that the localized high-frequency component $U_{j,k}^{\sigma,high}(t)$ of the half-wave propagator satisfies the following estimate.
\begin{proposition}\label{prop:high-frequency-half-wave-strichartz}
Let $U_{j,k}^{\sigma,high}(t)$ be as in \cref{sec:preparation}. 
Then, for any $(q, r) \in [2, \infty] \times (2, \infty]$ and any $\sigma \in \rR$, we have  
\begin{Eq*}
\kl\|U_{j,k}^{\sigma,high}(t)f\kr\|_{L_t^qL_z^r}\lesssim (1+j)2^{j(s-\sigma)}\|f\|_{L^2},
\end{Eq*}
where $s = s_e$, as defined in \eqref{eq:euclidean-admissible-triples}, when $2/q \geq (n-1)(1/2 - 1/r)$; and $s = s_w$, as defined in \eqref{eq:wave-admissible-triples}, when $2/q \leq (n-1)(1/2 - 1/r)$.
\end{proposition}

\subsection{Proof of Proposition~\ref{prop:high-frequency-half-wave-strichartz}}
Following the $TT^*$ method, to prove this proposition, it suffices to show that for any $(q, r)$ as above and any functions $F, G$, the following estimate holds:
\begin{Eq}\label{eq:high-frequency-half-wave-bilinear-estimate}
&\kl|\iint\kl<[U_{j,k}^{\sigma,high}]^*(\tau)F(\tau),[U_{j,k}^{\sigma,high}]^*(t)G(t)\kr>\d\tau\d t\kr|\\
\lesssim&(1+j)^22^{2j(s-\sigma)}\|F(\tau)\|_{L_\tau^{q'}L_z^{r'}}\|G(t)\|_{L_t^{q'}L_z^{r'}},
\end{Eq}
where $s$ is defined as above.

This estimate, in turn, is a consequence of the following lemma.
\begin{lemma}\label{lem:high-frequency-half-wave-kernel-estimate}
For the operator $U_{j,k}^{\sigma,high}(t)[U_{j,k}^{\sigma,high}]^*(\tau)$, any fixed $0 < \delta \ll 1$ and $M>0$ sufficiently large, we have the following kernel estimate:  
\begin{Eq}\label{eq:high-frequency-half-wave-kernel-estimate}
I(t-\tau):=&\ker\kl(U_{j,k}^{\sigma,high}(t)[U_{j,k}^{\sigma,high}]^*(\tau)\kr)\\
=&\int_0^\infty e^{\II (t-\tau)\sqrt{\lambda^2+\kappa^2}}\beta_j^2(\lambda)\sqrt{\lambda^2+\kappa^2}^{-2\sigma}[Q_kE_{\sqrt{H_0}}'Q_k^*](\lambda)\d\lambda,\\
\kl|I(t)\kr|\lesssim&2^{j(\frac{n+1}{2}-2\sigma)}e^{-(\frac{n-1}{2}-\delta)d(z,z')}\begin{cases}(2^{-j}+|t|)^{-\frac{n-1}{2}},&|t|\lesssim 1,\\
|t|^{-M},&|t|\gtrsim 1.
\end{cases}
\end{Eq}
\end{lemma}

Assuming the above lemma holds, 
we immediately obtain the $L^1 \to L^\infty$ bound for the operator $U_{j,k}^{\sigma,high}(t)[U_{j,k}^{\sigma,high}]^*(\tau)$.
On the other hand, the argument of \cite[Proposition~3.1]{MR4169670} applies without change: on the high-frequency support, $\sqrt{\lambda^2+\kappa^2}^{-\sigma}$ is a symbol of order $-\sigma$.  Hence the $L^2\rightarrow L^2$ bound of this operator is controlled by $2^{-2\sigma j}$.
Interpolating between these two bounds, we obtain for any $r \in [2, \infty]$ and sufficiently large $M$,
\begin{Eq*}
\kl\|U_{j,k}^{\sigma,high}(t)[U_{j,k}^{\sigma,high}]^*(\tau)\kr\|_{L_z^{r'}\rightarrow L_z^r}
\lesssim&2^{j\kl((n+1)\kl(\frac{1}{2}-\frac{1}{r}\kr)-2\sigma\kr)}\begin{cases}(2^{-j}+|t-\tau|)^{-(n-1)\kl(\frac{1}{2}-\frac{1}{r}\kr)},&|t-\tau|\lesssim 1,\\
|t-\tau|^{2M/r-M},&|t-\tau|\gtrsim 1.
\end{cases}
\end{Eq*}
Now, for any $r\in(2,\infty]$, we fix $\delta\ll 1$ and $M$ large enough.
Then we obtain
\begin{Eq*}
&\kl|\iint\kl<[U_{j,k}^{\sigma,high}]^*(\tau)F(\tau),[U_{j,k}^{\sigma,high}]^*(t)G(t)\kr>\d\tau\d t\kr|\\
\lesssim&2^{j\kl((n+1)\kl(\frac{1}{2}-\frac{1}{r}\kr)-2\sigma\kr)}
\|F(\tau)\|_{L_\tau^{q'}L_z^{r'}}\|G(t)\|_{L_t^{q'}L_z^{r'}}\\
&\quad\cdot\kl(\kl\|(2^{-j}+|t|)^{-(n-1)\kl(\frac{1}{2}-\frac{1}{r}\kr)}\kr\|_{L_t^{q/2}(|t|\lesssim1)}
+\kl\||t|^{2M/r-M}\kr\|_{L_t^{q/2}(|t|\gtrsim1)}\kr)\\
\lesssim&2^{j\kl((n+1)\kl(\frac{1}{2}-\frac{1}{r}\kr)-2\sigma\kr)}\|F(\tau)\|_{L_\tau^{q'}L_z^{r'}}\|G(t)\|_{L_t^{q'}L_z^{r'}}\begin{cases}
2^{j\kl((n-1)\kl(\frac{1}{2}-\frac{1}{r}\kr)-\frac{2}{q}\kr)},&(n-1)\kl(\frac{1}{2}-\frac{1}{r}\kr)>\frac{2}{q},\\
j+1,&(n-1)\kl(\frac{1}{2}-\frac{1}{r}\kr)=\frac{2}{q},\\
1,&(n-1)\kl(\frac{1}{2}-\frac{1}{r}\kr)<\frac{2}{q}.
\end{cases}\\
\lesssim&\|F(\tau)\|_{L_\tau^{q'}L_z^{r'}}\|G(t)\|_{L_t^{q'}L_z^{r'}}(1+j)^2\begin{cases}
2^{2j\kl(n\kl(\frac{1}{2}-\frac{1}{r}\kr)-\frac{1}{q}-\sigma\kr)},&(n-1)\kl(\frac{1}{2}-\frac{1}{r}\kr)\geq\frac{2}{q},\\
2^{2j\kl(\frac{n+1}{2}\kl(\frac{1}{2}-\frac{1}{r}\kr)-\sigma\kr)},&(n-1)\kl(\frac{1}{2}-\frac{1}{r}\kr)\leq\frac{2}{q}.
\end{cases}
\end{Eq*}
which yields precisely \eqref{eq:high-frequency-half-wave-bilinear-estimate}.
Consequently, to establish \cref{prop:high-frequency-half-wave-strichartz}, it remains only to verify \cref{lem:high-frequency-half-wave-kernel-estimate}.

\subsection{Proof of Lemma~\ref{lem:high-frequency-half-wave-kernel-estimate}}
Using \eqref{eq:microlocalized-spectral-measure-expansion}, and writing $d(z,z')$ simply as $d$, we can further decompose the $I(t)$ defined above into
\begin{Eq*}
I(t)=&\sum_{\pm}\int_{2^{j-1}}^{2^{j+1}}e^{\II \kl(t\sqrt{\lambda^2+\kappa^2}\pm \lambda d\kr)}\lambda^{n-1}\sqrt{\lambda^2+\kappa^2}^{-2\sigma}\beta_j^2(\lambda)a_{\pm}(\lambda)\d\lambda\\
&\quad+\int_{2^{j-1}}^{2^{j+1}}e^{\II t\sqrt{\lambda^2+\kappa^2}}\sqrt{\lambda^2+\kappa^2}^{-2\sigma}\beta_j^2(\lambda)b(\lambda)\d\lambda\\
\equiv&I_{a_{\pm}}+I_b.
\end{Eq*}
For the terms $I_{a_{\pm}}$ and $I_b$, we present the following claims, 
analogous to the analysis in the previous section:
\begin{claim}\label{claim:high-frequency-oscillatory-integral-expansion}
For any fixed $M\in\rN_0$, the term $I_{a_{\pm}}(t)$ can be rewritten as:
\begin{Eq*}
I_{a_{\pm}}=&\sum_{\bf M}C_{\bf M}\int_{2^{j-1}}^{2^{j+1}}e^{\II \kl(t\sqrt{\lambda^2+\kappa^2}\pm \lambda d\kr)} \lambda^{n-1-M_1+M_2+M_{52}}\sqrt{\lambda^2+\kappa^2}^{M-2\sigma-2M_2-M_{52}}\\
&\phantom{\sum_{\bf M}C_{\bf M}\int_{2^{j-1}}^{2^{j+1}}}\quad\cdot[\beta_j^2]^{(M_3)}(\lambda)a_{\pm}^{(M_4)}(\lambda)\kl(t\lambda\pm \sqrt{\lambda^2+\kappa^2}d\kr)^{-M-M_5}|t|^{M_{51}}d^{M_{52}}\d\lambda,
\end{Eq*}
where ${\bf M}=(M_1,M_2,M_3,M_4,M_5,M_{51},M_{52})$, whose components are all non-negative and satisfy $M=\sum_{i=1}^5 M_i$ and $M_5=M_{51}+M_{52}$.

Similarly, $I_{b}$ can be rewritten as:
\begin{Eq*}
I_{b}=&|t|^{-M}\sum_{\bf M}C_{\bf M}\int_{2^{j-1}}^{2^{j+1}}e^{\II t\sqrt{\lambda^2+\kappa^2}}
 \lambda^{-M-M_1+M_2}\sqrt{\lambda^2+\kappa^2}^{M-2\sigma-2M_2}[\beta_j^2]^{(M_3)}(\lambda)b^{(M_4)}(\lambda)\d\lambda,
\end{Eq*}
where ${\bf M}=(M_1,M_2,M_3,M_4)$, whose components are all non-negative and satisfy $M=\sum_{i=1}^4 M_i$.
\end{claim}

Under this claim, noting that $[\beta_j^2]^{(M_3)}(\lambda) \lesssim 2^{-M_3 j}$ and $\sqrt{\lambda^2 + \kappa^2} \approx \lambda \approx 2^j$, we first obtain
\begin{Eq*}
|I_{a_{\pm}}(t)|\lesssim&\sum_{\bf M}2^{j(n-1-M_1-M_2-M_3+M-2\sigma)}|t|^{M_{51}}d^{M_{52}}\int_{2^{j-1}}^{2^{j+1}} \kl|a_{\pm}^{(M_4)}(\lambda)\kr|\kl|t\lambda\pm \sqrt{\lambda^2+\kappa^2}d\kr|^{-M-M_5}\d\lambda\\
\lesssim&\sum_{\bf M}2^{j(n-1+M_4+M_5-2\sigma)}\max\{|t|,d\}^{M_5}\int_{2^{j-1}}^{2^{j+1}} \kl|a_{\pm}^{(M_4)}(\lambda)\kr|\kl|t\lambda\pm \sqrt{\lambda^2+\kappa^2}d\kr|^{-M-M_5}\d\lambda.
\end{Eq*}

To estimate $|I_{a_{\pm}}(t)|$, 
we proceed by considering separate cases based on the sizes of $d:=d(z,z')$ and $|t|$.\\
In the near-stationary regime $|t|\approx d$, we do not integrate by parts in the oscillatory factor; instead we use the dyadic size estimate and absorb polynomial losses into the exponential spatial decay.  Indeed, for every $N>0$ and every small $\delta>0$,
\begin{Eq*}
e^{-\frac{n-1}{2}d}
=e^{-(\frac{n-1}{2}-\delta)d}e^{-\delta d}
\lesssim e^{-(\frac{n-1}{2}-\delta)d}(1+|t|)^{-N}.
\end{Eq*}
\noindent  {\bf Case 1: $d\lesssim 1$.} 
In this case, by applying  \eqref{eq:oscillatory-amplitude-estimate}, we obtain
\begin{Eq*}
|I_{a_{\pm}}(t)|
\lesssim&(1+2^j d)^{-\frac{n-1}{2}}\sum_{\bf M}2^{j(n-1+M_5-2\sigma)}\max\{|t|,d\}^{M_5}\int_{2^{j-1}}^{2^{j+1}} \kl|t\lambda\pm \sqrt{\lambda^2+\kappa^2}d\kr|^{-M-M_5}\d\lambda.
\end{Eq*}

\noindent {\bf Case 1.1: Additionally $|t|\lesssim d$.} 
In this situation, we take $M=0$, which gives
\begin{Eq*}
|I_{a_{\pm}}(t)|
\lesssim&(1+2^j d)^{-\frac{n-1}{2}} 2^{j(n-2\sigma)}\lesssim(1+2^j |t|)^{-\frac{n-1}{2}} 2^{j(n-2\sigma)}\lesssim 2^{j(\frac{n+1}{2}-2\sigma)}(2^{-j}+ |t|)^{-\frac{n-1}{2}}.
\end{Eq*}

\noindent {\bf Case 1.2: Additionally $|t|\gg d$.} 
Here, for any $M \in \rN_0$, we have
\begin{Eq*}
|I_{a_{\pm}}(t)|
\lesssim&\sum_{\bf M}2^{j(n-1+M_5-2\sigma)}|t|^{M_5}\int_{2^{j-1}}^{2^{j+1}} \kl|t\lambda \kr|^{-M-M_5}\d\lambda\lesssim 2^{j(n-2\sigma-M)}|t|^{-M}.
\end{Eq*}
Thus, when $|t| \gtrsim 1$, we choose $M > \frac{n-1}{2}$ to ensure that
$|I_{a_{\pm}}(t)|\lesssim 2^{j(\frac{n+1}{2}-2\sigma)}|t|^{-M}$.
On the other hand, for $|t|\lesssim 1$, we take $M=0$ and $M=n>\frac{n-1}{2}$, respectively.
This yields
\begin{Eq*}
|I_{a_{\pm}}(t)|
\lesssim&\min\{2^{j(n-2\sigma)},2^{-2j\sigma}|t|^{-n}\}\lesssim 2^{j(\frac{n+1}{2}-2\sigma)}(2^{-j}+ |t|)^{-\frac{n-1}{2}}.
\end{Eq*}

\noindent  {\bf Case 2: $d\gtrsim 1$.} 
In this case, by applying  \eqref{eq:oscillatory-amplitude-estimate}, we have
\begin{Eq*}
|I_{a_{\pm}}(t)|
\lesssim& e^{-\frac{n-1}{2}d}\sum_{\bf M}2^{j(\frac{n-1}{2}+M_5-2\sigma)}\max\{|t|,d\}^{M_5}\int_{2^{j-1}}^{2^{j+1}} \kl|t\lambda\pm \sqrt{\lambda^2+\kappa^2}d\kr|^{-M-M_5}\d\lambda.
\end{Eq*}

\noindent  {\bf Case 2.1: Additionally $|t|\gg d$.}
In this case, fixing $M$ sufficiently large, we obtain
\begin{Eq*}
|I_{a_{\pm}}(t)|
\lesssim& e^{-\frac{n-1}{2}d}\sum_{\bf M}2^{j(\frac{n-1}{2}+M_5-2\sigma)}|t|^{M_5}\int_{2^{j-1}}^{2^{j+1}} \kl|t\lambda\kr|^{-M-M_5}\d\lambda\\
\lesssim &e^{-\frac{n-1}{2}d} 2^{j(\frac{n+1}{2}-M-2\sigma)}|t|^{-M}\lesssim 2^{j(\frac{n+1}{2}-2\sigma)}|t|^{-M}e^{-(\frac{n-1}{2}-\delta)d}.
\end{Eq*}

\noindent  {\bf Case 2.2: Additionally $|t|\approx d$.} 
Here we do not integrate by parts.  The dyadic size estimate gives
\begin{Eq*}
|I_{a_{\pm}}(t)|
\lesssim& e^{-\frac{n-1}{2}d}2^{j(\frac{n-1}{2}-2\sigma)}\int_{2^{j-1}}^{2^{j+1}} \d\lambda
\lesssim 2^{j(\frac{n+1}{2}-2\sigma)}e^{-\frac{n-1}{2}d}.
\end{Eq*}
Since $|t|\approx d\gtrsim1$, for any $M>0$ we have
\begin{Eq*}
e^{-\frac{n-1}{2}d}\lesssim |t|^{-M}e^{-(\frac{n-1}{2}-\delta)d}.
\end{Eq*}
This gives the desired estimate in this case.

\noindent  {\bf Case 2.3: Additionally $|t|\ll d$.}
In this case, we first observe that for any $M \in \rN_0$,
\begin{Eq*}
|I_{a_{\pm}}(t)|
\lesssim e^{-\frac{n-1}{2}d}\sum_{\bf M}2^{j(\frac{n-1}{2}+M_5-2\sigma)}d^{M_5}\int_{2^{j-1}}^{2^{j+1}} \kl|\sqrt{\lambda^2+\kappa^2}d\kr|^{-M-M_5}\d\lambda
\lesssim &e^{-\frac{n-1}{2}d} 2^{j(\frac{n+1}{2}-M-2\sigma)}d^{-M}.
\end{Eq*}
When $|t| \gtrsim 1$, it clearly follows that
\begin{Eq*}
|I_{a_{\pm}}(t)|
\lesssim & 2^{j(\frac{n+1}{2}-2\sigma)}|t|^{-M} e^{-(\frac{n-1}{2}-\delta)d}.
\end{Eq*}
When $|t| \lesssim 1$, 
using the fact that $d \gtrsim 2^{-j} + |t|$ and taking $M > \frac{n-1}{2}$, 
we also have
\begin{Eq*}
|I_{a_{\pm}}(t)|
\lesssim 2^{j(\frac{n+1}{2}-M-2\sigma)}(2^{-j}+|t|)^{-M}e^{-\frac{n-1}{2}d}
\lesssim  2^{j(\frac{n+1}{2}-2\sigma)}(2^{-j}+|t|)^{-\frac{n-1}{2}}e^{-(\frac{n-1}{2}-\delta)d}.
\end{Eq*}

Combining all the cases above,
we have shown that $I_{a_{\pm}}(t)$ satisfies the estimates required for $I(t)$ in \eqref{eq:high-frequency-half-wave-kernel-estimate}.
As for $I_b(t)$, by applying \cref{claim:high-frequency-oscillatory-integral-expansion} and \eqref{eq:remainder-amplitude-estimate}, we similarly obtain, for any $M' > 0$,
\begin{Eq*}
|I_b(t)|\lesssim& 2^{j(1-2\sigma-M')}e^{-\frac{n-1}{2}d},\qquad |t|\lesssim1,\\
|I_{b}(t)|\lesssim &|t|^{-M}2^{j(1-2\sigma-M-M')}e^{-\frac{n-1}{2}d},
\qquad |t|\gtrsim1,
\end{Eq*}
Choosing $M'$ sufficiently large in the small-time estimate and $M$ sufficiently large in the large-time estimate, these bounds imply the desired estimate for $I(t)$ in \eqref{eq:high-frequency-half-wave-kernel-estimate}.
This completes the proof of \cref{lem:high-frequency-half-wave-kernel-estimate}, assuming the validity of \cref{claim:high-frequency-oscillatory-integral-expansion}.

\subsection{Proof of Claim~\ref{claim:high-frequency-oscillatory-integral-expansion}}
We prove \cref{claim:high-frequency-oscillatory-integral-expansion} by induction on $M$. 
We begin with the $I_{a_{\pm}}$.
The case $M=0$ is immediate.
Assume that the claim holds for some $M\geq0$.
We use the identity
\begin{Eq*}
e^{\II \kl(t\sqrt{\lambda^2+\kappa^2}\pm \lambda d\kr)}=\kl[\frac{\sqrt{\lambda^2+\kappa^2}}{\II\kl(t\lambda\pm \sqrt{\lambda^2+\kappa^2} d\kr)}\partial_\lambda\kr]\kl(e^{\II \kl(t\sqrt{\lambda^2+\kappa^2}\pm \lambda d\kr)}\kr),
\end{Eq*}
The compact support of the dyadic cutoffs removes all boundary contributions in the integrations by parts.
Therefore, we have
\begin{Eq*}
&\int_{2^{j-1}}^{2^{j+1}}e^{\II \kl(t\sqrt{\lambda^2+\kappa^2}\pm \lambda d\kr)} \lambda^{n-1-M_1+M_2+M_{52}}\sqrt{\lambda^2+\kappa^2}^{M-2\sigma-2M_2-M_{52}}\\
&\phantom{\int_{2^{j-1}}^{2^{j+1}}}\quad\cdot[\beta_j^2]^{(M_3)}(\lambda)a_{\pm}^{(M_4)}(\lambda)\kl(t\lambda\pm \sqrt{\lambda^2+\kappa^2}d\kr)^{-M-M_5}|t|^{M_{51}}d^{M_{52}}\d\lambda\\
=&-\II\int_{2^{j-1}}^{2^{j+1}}\partial_\lambda\kl(e^{\II \kl(t\sqrt{\lambda^2+\kappa^2}\pm \lambda d\kr)}\kr) \lambda^{n-1-M_1+M_2+M_{52}}\sqrt{\lambda^2+\kappa^2}^{(M+1)-2\sigma-2M_2-M_{52}}\\
&\phantom{\int_{2^{j-1}}^{2^{j+1}}}\quad\cdot[\beta_j^2]^{(M_3)}(\lambda)a_{\pm}^{(M_4)}(\lambda)\kl(t\lambda\pm \sqrt{\lambda^2+\kappa^2}d\kr)^{-(M+1)-M_5}|t|^{M_{51}}d^{M_{52}}\d\lambda.
\end{Eq*}
Next, we perform integration by parts. 
When $\partial_\lambda$ falls on $\lambda^{a}$, 
we get $a \lambda^{a-1}$, 
which corresponds to $M_1 \to M_1 + 1$;
when $\partial_\lambda$ falls on $\sqrt{\lambda^2+\kappa^2}^{a}$, 
we get $a \lambda \sqrt{\lambda^2+\kappa^2}^{a-2}$, 
which corresponds to $M_2 \to M_2 + 1$;
when $\partial_\lambda$ falls on $[\beta_j^2]^{(M_3)}(\lambda)$ or $a_{\pm}^{(M_4)}(\lambda)$,
we get $[\beta_j^2]^{(M_3+1)}(\lambda)$ or $a_{\pm}^{(M_4+1)}(\lambda)$ separately,
which correspond to $M_3 \to M_3 + 1$ and $M_4 \to M_4 + 1$, respectively;
finally, when $\partial_\lambda$ falls on $\kl(t\lambda\pm \sqrt{\lambda^2+\kappa^2}d\kr)^{a}$,
we get $\kl(t\lambda\pm \sqrt{\lambda^2+\kappa^2}d\kr)^{a-1}\kl(t\pm\frac{\lambda d}{\sqrt{\lambda^2+\kappa^2}}\kr)$,
which corresponds to $M_{51}\rightarrow M_{51}+1$ or $M_{52}\rightarrow M_{52}+1$, and then $M_5\rightarrow M_5+1$.
This completes the induction.

Now we deal with $I_b$.
The case $M=0$ is again immediate.
Assume that the claim holds for some $M\geq0$.
We use the identity
\begin{Eq*}
e^{\II t\sqrt{\lambda^2+\kappa^2}}=\kl[\frac{\sqrt{\lambda^2+\kappa^2}}{\II t\lambda}\partial_\lambda\kr]\kl(e^{\II t\sqrt{\lambda^2+\kappa^2}}\kr),
\end{Eq*}
Therefore, we have
\begin{Eq*}
&|t|^{-M}\sum_{\bf M}C_{\bf M}\int_{2^{j-1}}^{2^{j+1}}e^{\II t\sqrt{\lambda^2+\kappa^2}}
 \lambda^{-M-M_1+M_2}\sqrt{\lambda^2+\kappa^2}^{M-2\sigma-2M_2}[\beta_j^2]^{(M_3)}(\lambda)b^{(M_4)}(\lambda)\d\lambda\\
 =&-\II |t|^{-(M+1)}\sum_{\bf M}C_{\bf M}\int_{2^{j-1}}^{2^{j+1}}\partial_\lambda\kl(e^{\II t\sqrt{\lambda^2+\kappa^2}}\kr)
 \lambda^{-(M+1)-M_1+M_2}\sqrt{\lambda^2+\kappa^2}^{(M+1)-2\sigma-2M_2}\\
&\phantom{-\II |t|^{-(M+1)}\sum_{\bf M}C_{\bf M}\int_{2^{j-1}}^{2^{j+1}}}\quad\cdot [\beta_j^2]^{(M_3)}(\lambda)b^{(M_4)}(\lambda)\d\lambda.
\end{Eq*}
Next, we perform integration by parts. 
When $\partial_\lambda$ falls on the last four terms,
we get terms corresponding to $M_i \to M_i + 1$ with $i=1,2,3,4$, respectively.
This completes the induction.

\section{Strichartz estimates for the Klein-Gordon equation}\label{sec:strichartz-estimates-for-main-equation}
In this section, we prove \cref{thm:main-strichartz-estimate} by combining \cref{prop:low-frequency-half-wave-strichartz} and \cref{prop:high-frequency-half-wave-strichartz}.
We first reduce the spacetime part of the theorem to
\begin{Eq}\label{eq:spacetime-strichartz-reduction}
\|\phi\|_{L_t^q(I;L_z^r(M^\circ))}
\lesssim \|\psi_0\|_{H^{\mu}(M^\circ)}+\|\psi_1\|_{H^{\mu-1}(M^\circ)}+\|F\|_{L_t^{\tilde q'}(I;H_{\tr'}^{\mu+\tilde\mu-1}(M^\circ))},
\end{Eq}
for $(q,r,\mu),(\tilde q,\tr,\tilde \mu)\in \Lambda_w\cup \Lambda_e$, 
defined in \eqref{eq:wave-admissible-triples} and \eqref{eq:euclidean-admissible-triples}.

Recalling that $H_\kappa = -\Delta_g - \frac{(n-1)^2}{4} + \kappa^2$, the solution to \eqref{eq:positive-mass-wave-equation} can be written as
\begin{Eq*}
\phi=&\phi_{h}+\phi_{i},\quad \phi_{h}=\frac{U^0(t)+U^{0}(-t)}{2}\psi_0+\frac{U^0(t)-U^{0}(-t)}{2\II\sqrt{H_\kappa}}\psi_1,\\
\phi_{i}=&\int_{0\leq\tau<t}\frac{U^0(t)U^0(\tau)^*-U^0(-t)U^0(-\tau)^*}{2\II \sqrt{H_\kappa}}F(\tau)\d\tau.
\end{Eq*}
We estimate the homogeneous part $\phi_h$ and the inhomogeneous part $\phi_i$ separately, which completes the proof of \cref{thm:main-strichartz-estimate}.

\subsection{Homogeneous Strichartz estimates}
In this subsection, 
we establish the Strichartz estimate for the homogeneous part of the solution:
\begin{proposition}\label{prop:homogeneous-strichartz-estimate}
Let $\phi_h$ be defined as above. 
Then the following estimate holds:
\begin{Eq}\label{eq:homogeneous-strichartz-estimate}
\|\phi_h(t,z)\|_{L_t^q(I;L_z^r(M^\circ))}\lesssim \|\psi_0\|_{H^\mu(M^\circ)}+\|\psi_1\|_{H^{\mu-1}(M^\circ)},
\end{Eq}
with $(q,r,\mu)\in \Lambda_w\cup \Lambda_e$.
\end{proposition}

Using the expression for $\phi_h$, we first write
\begin{Eq*}
\phi_h=&\sum_{\pm}\sum_{\sigma\in\{0,1\}}C_{\pm,\sigma}U^{\sigma}(\pm t)\psi_\sigma=\sum_{\pm}\sum_{\sigma\in\{0,1\}}C_{\pm,\sigma}U^{\mu}(\pm t)\sqrt{H_\kappa}^{\mu-\sigma}\psi_\sigma.
\end{Eq*}
Thus, without loss of generality, it suffices to prove that for any $f(z)$,
\begin{Eq*}
\|U^{\mu}(t)f\|_{L_t^qL_z^r}\lesssim \|f\|_{L^2}.
\end{Eq*}
For the low-frequency part, the desired estimate follows directly from \cref{prop:low-frequency-half-wave-strichartz}.
For the high-frequency part, by applying \cref{prop:high-frequency-half-wave-strichartz}, we obtain
\begin{Eq*}
\|U^{\mu,high}(t)f\|_{L_t^qL_z^r}\lesssim\sum_{k=1}^K\sum_{j=1}^\infty\|U_{j,k}^{\mu,high}(t)f\|_{L_t^qL_z^r}
\lesssim  \sum_{k=1}^K\sum_{j=1}^\infty(1+j)2^{j(s-\mu)}\|f\|_{L^2}.
\end{Eq*}
Since $\mu > s$, the sum over $j$ converges, and the proof of \cref{prop:homogeneous-strichartz-estimate} is complete.

The energy part of \eqref{eq:main-strichartz-estimate} is also controlled by the same homogeneous estimate and its dual.
The homogeneous Strichartz estimate is uniform on time intervals, and therefore its dual estimate is also uniform on subintervals of $I$.
Indeed, the dual form of \cref{prop:homogeneous-strichartz-estimate}, applied with $(\tilde q,\tr,\tilde\mu)$, gives
\begin{Eq*}
\kl\|\int_I e^{-\II \tau\sqrt{H_\kappa}}H_\kappa^{\frac{\mu-1}{2}}F(\tau)\d\tau\kr\|_{L_z^2}
\lesssim \|H_\kappa^{\frac{\mu+\tilde\mu-1}{2}}F\|_{L_t^{\tilde q'}(I;L_z^{\tr'})}.
\end{Eq*}
Consequently, the same bound holds with $I$ replaced by $I\cap(-\infty,t)$, uniformly in $t\in I$.
Using the half-wave representation of the Duhamel term, this yields
\begin{Eq*}
\|\phi_i\|_{L_t^\infty(I;H^\mu)}+\|\partial_t\phi_i\|_{L_t^\infty(I;H^{\mu-1})}
\lesssim \|F\|_{L_t^{\tilde q'}(I;H_{\tr'}^{\mu+\tilde\mu-1})}.
\end{Eq*}
Together with the elementary conservation bound for the homogeneous part, this proves the energy component of \eqref{eq:main-strichartz-estimate}.

\subsection{Inhomogeneous Strichartz estimates for non-double endpoint}
In this subsection, we establish the Strichartz estimate for the inhomogeneous part of the solution in the case $q > \tilde{q}'$.
Specifically, we prove the following:
\begin{proposition}\label{prop:inhomogeneous-strichartz-non-double-endpoint}
Let $\phi_i$ be defined as above. Then the following estimate holds:
\begin{Eq}\label{eq:inhomogeneous-strichartz-non-double-endpoint}
\|\phi_i(t,z)\|_{L_t^q(I;L_z^r(M^\circ))}\lesssim \|F\|_{L_t^{\tilde q'}(I;H_{\tr'}^{\mu+\tilde\mu-1}(M^\circ))},
\end{Eq}
with $(q,r,\mu),(\tilde q,\tr,\tilde \mu)\in \Lambda_w\cup \Lambda_e$ and $q>\tilde q'$.
\end{proposition}

Based on the expression for $\phi_i$, we begin with
\begin{Eq*}
\phi_{i}=&\sum_{\pm}C_{\pm}\int_{\tau<t}\frac{U^0(\pm t)U^0(\pm\tau)^*}{\sqrt{H_\kappa}}F(\tau)\d\tau\\
=&\sum_{\pm}C_{\pm}\int_{\tau<t}U^\mu(\pm t)U^{\tilde\mu}(\pm\tau)^*\sqrt{H_\kappa}^{\mu+\tilde \mu-1}F(\tau)\d\tau.
\end{Eq*}
Hence, without loss of generality, it suffices to prove that for any $G(\tau, z)$,
\begin{Eq*}
\kl\|\int_{\tau<t}U^\mu(t)U^{\tilde\mu}(\tau)^*G(\tau)\d\tau\kr\|_{L_t^qL_z^r}\lesssim \|G(\tau)\|_{L_\tau^{\tilde q'}L_z^{\tr'}}.
\end{Eq*}

For the low-low frequency part, applying \cref{prop:low-frequency-half-wave-strichartz} and its dual form, we immediately obtain
\begin{Eq*}
\kl\|\int_\rR U^{\mu,low}(t)U^{\tilde\mu,low}(\tau)^*G(\tau)\d\tau\kr\|_{L_t^qL_z^r}\lesssim \|G(\tau)\|_{L_\tau^{\tilde q'}L_z^{\tr'}}.
\end{Eq*}
Here and below, $s(q,r)$ denotes $s_w$ or $s_e$ according as $(q,r)$ lies in the wave or Euclidean admissible region; on the common boundary the two definitions agree.
For the high-high frequency part, using \cref{prop:high-frequency-half-wave-strichartz} together with its dual, we similarly obtain
\begin{Eq*}
&\kl\|\int_\rR U^{\mu,high}(t)U^{\tilde\mu,high}(\tau)^*G(\tau)\d\tau\kr\|_{L_t^qL_z^r}\\
\lesssim& \sum\Low{k,k'=1}^K\sum_{j,j'=1}^\infty\kl\|\int_\rR U_{j,k}^{\mu,high}(t)U_{j',k'}^{\tilde\mu,high}(\tau)^*G(\tau)\d\tau\kr\|_{L_t^qL_z^r}\\
\lesssim& \sum\Low{k,k'=1}^K\sum_{j,j'=1}^\infty (1+j)(1+j')2^{j(s(q,r)-\mu)}2^{j'(s(\tilde q,\tr)-\tilde \mu)}\kl\|G(\tau)\kr\|_{L_t^{\tilde q'}L_z^{\tr'}}\\
\lesssim&\kl\|G(\tau)\kr\|_{L_t^{\tilde q'}L_z^{\tr'}}.
\end{Eq*}
Here no orthogonality between different dyadic-microlocal pieces is used.  We estimate each operator-valued piece by \cref{prop:high-frequency-half-wave-strichartz} and its dual, and then sum absolutely; the convergence follows from the strict inequalities $\mu>s(q,r)$ and $\tilde\mu>s(\tilde q,\tr)$.  The dependence of $Q_k(\lambda)$ on the spectral parameter is harmless here, since the microlocal partition is finite and uniformly $L^2$-bounded.
For the cross-frequency terms, note that $\chi(\lambda)\beta_j(\lambda)\neq0$ only for finitely many dyadic indices $j$.
Thus all interactions involving one low-frequency cutoff and one high-frequency cutoff reduce to finitely many compact spectral pieces and are controlled by the same estimates, together with the uniform $L^2$ boundedness of the microlocal cutoffs, with no dyadic summation loss.

Combining all the estimates above, we obtain
\begin{Eq*}
\kl\|\int_\rR U^\mu(t)U^{\tilde\mu}(\tau)^*G(\tau)\d\tau\kr\|_{L_t^qL_z^r}\lesssim \|G(\tau)\|_{L_\tau^{\tilde q'}L_z^{\tr'}}.
\end{Eq*}
Since admissibility gives $q\ge2$ and $\tilde q\ge2$, we have $\tilde q'\le2$.
Thus the condition $q>\tilde q'$ fails only when $q=\tilde q'=2$, equivalently $q=\tilde q=2$.
In the present non-double endpoint case, applying the Christ-Kiselev lemma completes the proof.

\subsection{Inhomogeneous Strichartz estimates for double endpoint}
In this subsection, we establish the Strichartz estimate for the inhomogeneous part of the solution in the double endpoint case $q = \tilde{q}' = 2$:
\begin{proposition}\label{prop:inhomogeneous-strichartz-double-endpoint}
Let $\phi_i$ be defined as above.  We obtain
\begin{Eq}\label{eq:inhomogeneous-strichartz-double-endpoint}
\|\phi_i(t,z)\|_{L_t^2(I;L_z^r(M^\circ))}\lesssim \|F\|_{L_t^{2}(I;H_{\tr'}^{\mu+\tilde\mu-1}(M^\circ))},
\end{Eq}
with $(2,r,\mu),(2,\tr,\tilde \mu)\in \Lambda_w\cup \Lambda_e$.
\end{proposition}

The high-frequency double-endpoint estimate follows from the endpoint summation argument in \cite[Section~5.2, Proposition~5.2]{MR4169670}.
That argument is designed precisely for the case where the Christ-Kiselev lemma cannot be applied.
The inputs used there are the dyadic $L^2\to L^2$ boundedness, the microlocalized high-frequency kernel bound, the rapid decay away from the relevant time scale, and the summability of the dyadic-microlocal pieces.
Lemma \ref{lem:high-frequency-half-wave-kernel-estimate} supplies the corresponding kernel estimate for the present positive-mass phase $\sqrt{\lambda^2+\kappa^2}$, while \cref{prop:high-frequency-half-wave-strichartz} gives the required dyadic summation.
Consequently the same $TT^*$ summation applies here, with the shifted-wave phase in \cite{MR4169670} replaced by $\sqrt{\lambda^2+\kappa^2}$.
The positive-mass modification enters the off-diagonal nonstationary-phase argument.  For the same microlocal decomposition as in \cite[Section~5.2]{MR4169670}, the phase satisfies
\begin{Eq*}
\partial_\lambda\kl((t-\tau)\sqrt{\lambda^2+\kappa^2}+\lambda\Phi\kr)
=(t-\tau)\frac{\lambda}{\sqrt{\lambda^2+\kappa^2}}+\Phi.
\end{Eq*}
On the non-incoming pieces with $t>\tau$, and the non-outgoing pieces with $t<\tau$, $t-\tau$ and $\Phi$ have the same sign.  Since $\lambda/\sqrt{\lambda^2+\kappa^2}$ is bounded above and below by positive constants on the high-frequency support for fixed $\kappa>0$, the same one-sided integration-by-parts argument applies.

For the low-frequency part, it suffices to establish the estimate for an arbitrary $G(\tau, z)$:
\begin{Eq*}
\kl\|\int_{\tau<t}U^{\mu,low}(t)U^{\tilde\mu,low}(\tau)^*G(\tau)\d\tau\kr\|_{L_t^2L_z^r}\lesssim \|G(\tau)\|_{L_\tau^{2}L_z^{\tr'}}.
\end{Eq*}
To apply \cref{lem:low-frequency-half-wave-kernel-estimate} and \cref{prop:kunze-stein-convolution-inequality},
we observe that
\begin{Eq*}
U^{\mu,low}(t)U^{\tilde\mu,low}(\tau)^*
=
U^{(\mu+\tilde\mu)/2,low}(t)U^{(\mu+\tilde\mu)/2,low}(\tau)^*,
\end{Eq*}
and we first consider the case $r,\tr<\infty$.  Define
$\alpha := \frac{2 \min\{r, \tr\}}{r + \tr}$ and $\theta := \frac{r \tr}{r + \tr}$,
which satisfy $\alpha + \theta > 2$ for $r,\tr>2$.
Then, for sufficiently small $\delta > 0$, we have
\begin{Eq*}
&\kl\|U^{\mu,low}(t)U^{\tilde\mu,low}(\tau)^*G(\tau)\kr\|_{L_z^r}\\
\lesssim&(1+|t-\tau|)^{-\frac{3}{2}}\kl(\int_0^{\infty}(1+R)^{\alpha}e^{(n-1)\kl(1-\frac{\alpha}{2}-\frac{\theta}{2}\kr)R+\theta\delta R}\d R\kr)^{\frac{1}{\theta}}\|G(\tau)\|_{L_z^{\tr'}}\\
\lesssim&(1+|t-\tau|)^{-\frac{3}{2}}\|G(\tau)\|_{L_z^{\tr'}}.
\end{Eq*}
Moreover, applying Minkowski's inequality followed by Young's inequality gives
\begin{Eq*}
\kl\|\int_{\tau<t}U^{\mu,low}(t)U^{\tilde\mu,low}(\tau)^*G(\tau)\d\tau\kr\|_{L_t^2L_z^r}
\lesssim&\kl\|\int{\bf 1}_{>0}(t-\tau)(1+|t-\tau|)^{-\frac{3}{2}}\|G(\tau)\|_{L_z^{\tr'}}\d\tau\kr\|_{L_t^2}\\
\lesssim&\kl(\int(1+|t|)^{-\frac{3}{2}}\d t\kr)\|G(\tau)\|_{L_\tau^2L_z^{\tr'}}\\
\lesssim&\|G(\tau)\|_{L_\tau^2L_z^{\tr'}}.
\end{Eq*}
If one of the spatial exponents is infinite, the preceding Kunze-Stein computation is replaced by the corresponding pointwise kernel estimate and H\"older's inequality.  For instance, when $r=\infty$ and $\tr<\infty$, let $T_{t-\tau}$ denote the operator with kernel $K(t-\tau;z,z')$ satisfying \cref{lem:low-frequency-half-wave-kernel-estimate}.  Then
\begin{Eq*}
\|T_{t-\tau}f\|_{L_z^\infty}
\leq \sup_z \|K(t-\tau;z,\cdot)\|_{L_z^{\tr}}\|f\|_{L_z^{\tr'}}
\lesssim(1+|t-\tau|)^{-\frac{3}{2}}\|f\|_{L_z^{\tr'}}.
\end{Eq*}
The case $r<\infty$, $\tr=\infty$ follows by duality.  If $r=\tr=\infty$, the pointwise kernel bound gives directly the $L_z^1\to L_z^\infty$ estimate.  In all cases, the time convolution estimate above is unchanged.
This completes the proof.

\section{Some applications}\label{sec:applications}
In this section, we apply \cref{thm:main-strichartz-estimate} to a class of semilinear wave equations and establish global existence results for small initial data in suitable function spaces. Specifically, we prove \cref{thm:strauss-type-global-existence}, \cref{thm:glassey-type-global-existence}, and \cref{thm:smooth-nonlinearity-global-existence}.

\subsection{Proof of Theorem~\ref{thm:strauss-type-global-existence}}
The existence of a global solution follows directly from the standard iteration method.
We construct a sequence by setting $\phi_{-1} := 0$ and defining $\phi_k$ for $k \geq 0$ to satisfy
\begin{Eq*}
\begin{cases}
\kl(\partial_{t}^2-\Delta_g -\frac{(n-1)^2}{4}+\kappa^2\kr) \phi_k=F_p(\phi_{k-1}); \\ (\phi_k,\partial_t\phi_k)_{t=0}=\varepsilon(\psi_0,\psi_1).
\end{cases}
\end{Eq*}
Choose $(q,r,\mu)=(\tq,\tr,\tmu)=(p+1,p+1,1/2)$ in \cref{thm:main-strichartz-estimate}.
It is straightforward to verify that $(p+1,p+1,1/2)\in\Lambda_e$ while $p\in(1,1+\frac{4}{n-1})$.
Therefore, we obtain the estimates
\begin{Eq*}
\|\phi_k\|_{L_{t,z}^{p+1}}\lesssim& \varepsilon(\|\psi_0\|_{H^{1/2}}+\|\psi_1\|_{H^{-1/2}})+\|F_p(\phi_{k-1})\|_{L_{t,z}^{\frac{p+1}{p}}}
\lesssim\varepsilon+\|\phi_{k-1}\|_{L_{t,z}^{p+1}}^p,\\
\|\phi_{k+1}-\phi_k\|_{L_{t,z}^{p+1}}\lesssim &\|F_p(\phi_k)-F_p(\phi_{k-1})\|_{L_{t,z}^{\frac{p+1}{p}}}\lesssim \kl(\max_{\sigma=0,1}\|\phi_{k-\sigma}\|_{L_{t,z}^{p+1}}\kr)^{p-1}\|\phi_k-\phi_{k-1}\|_{L_{t,z}^{p+1}}.
\end{Eq*}
It follows that for $0 < \varepsilon < \varepsilon_0$ with $\varepsilon_0$ sufficiently small,
the sequence $\|\phi_k\|_{L_{t,z}^{p+1}}$ is uniformly bounded and $\|\phi_{k+1}-\phi_k\|_{L_{t,z}^{p+1}}\leq \|\phi_{k}-\phi_{k-1}\|_{L_{t,z}^{p+1}}/2$.
This implies that  $\{\phi_k\}$ converges in $L_{t,z}^{p+1}$ to a solution $\phi$ of \eqref{eq:strauss-type-nonlinear-equation}, satisfying $\|\phi\|_{L_{t,z}^{p+1}}<\infty$.
The energy part of \cref{thm:main-strichartz-estimate} also gives $(\phi,\partial_t\phi)\in L_t^\infty H^{1/2}\times L_t^\infty H^{-1/2}$.
A similar argument yields uniqueness of the solution, which we omit here.

Next, we prove that ${\bphi}:=(\phi,\partial_t\phi)$ scatters in $\bH^{1/2}:=H^{1/2}\times H^{-1/2}$.
For any  $0<T_1<T_2$, 
applying the energy estimate in \cref{thm:main-strichartz-estimate} with $(q,r,\mu)=(\tq,\tr,\tmu)=(p+1,p+1,1/2)$, we obtain
\begin{Eq*}
\|L(-T_2)\bphi(T_2)-L(-T_1)\bphi(T_1)\|_{\bH^{1/2}}=&\|\bphi(T_2)-L(T_2-T_1)\bphi(T_1)\|_{\bH^{1/2}}\\
\lesssim& \sup_{t\in[0,T_2-T_1]}\|\bphi(t+T_1)-L(t)\bphi(T_1)\|_{\bH^{1/2}}\\
\lesssim&\|(\partial_t^2+H_\kappa)\phi(t+T_1)\|_{L_{t,z}^{\frac{p+1}{p}}([0,T_2-T_1]\times M^\circ)}\\
\lesssim&\|\phi\|_{L_{t,z}^{p+1}([T_1,T_2]\times M^\circ)}^p.
\end{Eq*}
This implies that
\begin{Eq*}
\sup_{T\geq 0}&\|L(-T)\bphi(T)\|_{\bH^{1/2}}\lesssim \|\bphi(0)\|_{\bH^{1/2}}+\sup_{T\geq 0}\|\phi\|_{L_{t,z}^{p+1}([0,T]\times M^\circ)}^p<\infty,\\
\lim_{T_1,T_2\rightarrow +\infty}&\|L(-T_2)\bphi(T_2)-L(-T_1)\bphi(T_1)\|_{\bH^{1/2}}\lesssim \lim_{T_1,T_2\rightarrow +\infty}\|\phi\|_{L_{t,z}^{p+1}([T_1,T_2]\times M^\circ)}^p=0.
\end{Eq*}
By the Cauchy criterion, 
we may define $(\psi_0^+,\psi_1^+):=\lim_{T\rightarrow+\infty} L(-T)\bphi(T)$, 
so that \eqref{eq:scattering-for-strauss-type-solution} holds. This completes the proof.

\subsection{Proof of Theorem~\ref{thm:glassey-type-global-existence}}\label{subsec:proof-of-glassey-type-global-existence}
To begin with, we construct a sequence by setting $\phi_{-1} := 0$ and defining $\phi_k$ for $k \geq 0$ to satisfy
\begin{Eq*}
\begin{cases}
\kl(\partial_{t}^2-\Delta_g -\frac{(n-1)^2}{4}+\kappa^2\kr) \phi_k=F_p(\partial_t\phi_{k-1}); \\ (\phi_k,\partial_t\phi_k)_{t=0}=\varepsilon(\psi_0,\psi_1).
\end{cases}
\end{Eq*}

In the remainder of the proof, we work in the $(t, z) = (t, r, \Theta)$ coordinate framework.
We begin by introducing the notation used in this section:
\begin{Eq*}
\tpartial f:=\kl(\partial_t f,\partial_r f+\frac{n-1}{2\tanh r}f,\spartial f\kr),
\end{Eq*}
where $\spartial f:=(\sinh r)^{-1}\,\nabla_\Theta f$ denotes the hyperbolically normalized angular derivative in the spherical variables, with $\nabla_\Theta$ taken with respect to the standard sphere metric.
We shall use the following positive-mass local energy estimate, whose proof is given in Appendix~\ref{subsec:proof-of-positive-mass-local-energy-estimate}.
\begin{proposition}\label{prop:positive-mass-local-energy-estimate}
Let $n\geq 3$, $\kappa>0$, and assume that $g$ satisfies \cref{hyp:small-asymptotic-decay}. 
Define the Klein-Gordon operator
\begin{Eq*}
 P_g:=\partial_t^2-\Delta_g-\frac{(n-1)^2}{4}+\kappa^2.
\end{Eq*}
Then, for any fixed $0<2a<b<1$, we have
\begin{Eq*}
\|f\|_{[ E\cap  {LE}]}\lesssim& \|\tpartial^{\leq 1} f(0)\|_{L_z^2}+\|r^{a-\frac{1}{2}}e^{-a r}f\|_{L_{t,z}^2}+\| P_g f\|_{[L_t^1L_z^2+{LE}^*]},
\end{Eq*}
for any smooth function $f$ whose spatial support is compact on each finite time interval, 
with implicit constant depending  on $a$, $\kappa$, and the constants in \cref{hyp:small-asymptotic-decay},
where
\begin{Eq*}
\|f\|_{ E}:=&\kl\|(\tpartial^{\leq 1},r^{-1}) f\kr\|_{L_t^\infty L_z^2}\approx \|\partial_t f\|_{L_t^\infty L_z^2}+\|f\|_{L_t^\infty H^1},\\
\|f\|_{ {LE}}:=&\|r^{a-\frac{1}{2}}e^{-a r}(\tpartial,r^{-1})f\|_{L_t^2 L_z^2},\quad 
\|F\|_{ {LE}^*}:=\|r^{\frac{1}{2}-a}e^{a r}F\|_{L_t^2 L_z^2}.
\end{Eq*}
\end{proposition}

The subsequent argument differs from that in \cite{DaiWang2026} in the treatment of the local energy estimate,
which, in our case, involves the additional lower-order term in the local energy estimate stated above.
However, thanks to the Strichartz estimates \cref{thm:main-strichartz-estimate} established in the present work, this term can be controlled.

Since the metric, the initial data, and the nonlinearity in \cref{thm:glassey-type-global-existence} are spherically symmetric, each Picard iterate remains spherically symmetric.
For the iteration, we measure the approximate solutions in the following energy-local energy norm:
\begin{Eq*}
\varE_k:=\|\partial_t^{\leq 1}\phi_k\|_{E\cap LE},
\end{Eq*}
with $a>0$ fixed below.
We show that 
for sufficiently small $\varepsilon$, 
the bound $\varE_k \lesssim \varepsilon$ holds uniformly, 
with the implicit constant independent of $k$.

First, observe that $\varE_{-1} = 0\lesssim\varepsilon$. 
Next, we assume that for some $k$, we have $\varE_{k-1} \lesssim \varepsilon$. Applying \cref{prop:positive-mass-local-energy-estimate} to $f=\phi_k$ and $f=\partial_t\phi_k$, and using $[\partial_t,P_g]=0$, we obtain
\begin{Eq*}
\varE_k\lesssim& \|\tpartial^{\leq 1}\partial_t^{\leq 1}\phi_k(0)\|_{L^2}
+\|r^{a-\frac{1}{2}}e^{-ar}\partial_t^{\leq 1}\phi_k\|_{L_{t,z}^2}
+\|e^{2ar}\partial_t^{\leq 1}F_p(\partial_t\phi_{k-1})\|_{L_{t,z}^2}\\
\equiv&I_1+I_2+I_3.
\end{Eq*}
Here we used the elementary bound
$\|G\|_{L_t^1L_z^2+LE^*}\leq \|G\|_{LE^*}\lesssim \|e^{2ar}G\|_{L_{t,z}^2}$,
which follows from the definition of $LE^*$ and the boundedness of $r^{\frac12-a}e^{-ar}$.
For $I_1$, we use the equation at $t=0$.
Since $\phi_k(0)=\varepsilon\psi_0$ and $\partial_t\phi_k(0)=\varepsilon\psi_1$, we have
\begin{Eq*}
\partial_t^2\phi_k(0)=-\varepsilon H_\kappa\psi_0+F_p(\varepsilon\psi_1).
\end{Eq*}
Thus the assumption $(\psi_0,\psi_1)\in H^2\times H^1$, together with the smallness of $\varepsilon$, gives $I_1\lesssim\varepsilon$.
When the linear estimates are applied to $\partial_t\phi_k$, the forcing term is $\partial_tF_p(\partial_t\phi_{k-1})$, and it is included in $I_3$ below.
For $I_2$, first fix $R_1>0$ sufficiently small.  In the region $r\leq R_1$, we have
\begin{Eq*}
I_{2;(r\leq R_1)}:=&\|r^{a-\frac{1}{2}}e^{-ar}\partial_t^{\leq 1}\phi_k\|_{L_{t,z}^2(r\leq R_1)}\\
\leq& R_1\|r^{a-\frac{1}{2}}e^{-ar}\cdot r^{-1}\partial_t^{\leq 1}\phi_k\|_{L_{t,z}^2(r\leq R_1)}
\leq R_1\varE_k,
\end{Eq*} 
which can be absorbed into the left-hand side by choosing $R_1$ small.
For the region where $r\geq R_1$, 
we use H\"older's inequality and \cref{thm:main-strichartz-estimate} with 
$(q,r,\mu)=(\tilde q, \tilde r,\tilde\mu)=(2,\frac{2n}{n-2a},\frac{1}{2})$, yielding
\begin{Eq*}
I_{2;(r\geq R_1)}:=&\|r^{a-\frac{1}{2}}e^{-ar}\partial_t^{\leq 1}\phi_k\|_{L_{t,z}^2(r\geq R_1)}
\lesssim  \|r^{a-\frac{1}{2}}e^{-ar}\|_{L^{n/a}(r\geq R_1)}\|\partial_t^{\leq 1}\phi_k\|_{L_t^2L^{2n/(n-2a)}}\\
\lesssim& \|\partial_t^{\leq 1}\phi_k(0)\|_{H^{1/2}}+\|\partial_t\partial_t^{\leq 1}\phi_k(0)\|_{H^{-1/2}}+\|(\partial_t^2+H_\kappa)\partial_t^{\leq 1}\phi_k\|_{L^2L^{2n/(n+2a)}}\\
\lesssim&\|\tpartial^{\leq 1}\partial_t^{\leq 1}\phi_k(0)\|_{L^2}+\|e^{2ar}\partial_t^{\leq 1}(\partial_t^2+H_\kappa)\phi_k\|_{L_{t,z}^2} \|e^{-2ar}\|_{L^{n/a}(r\geq R_1)}\\
\lesssim&I_1+I_3.
\end{Eq*}

Finally, for $I_3$, we again use H\"older's inequality.
Since $p \in (1, 1 + \frac{1}{n - 2})$, the quantities in the following minimum are positive.  We choose $a>0$ satisfying
\begin{Eq*}
a<\min\kl\{\frac{b}{2},\frac{p}{2}-\frac{(p-1)(n-1)}{2},\frac{(p-1)(n-1)}{6}\kr\}.
\end{Eq*}
We shall use the Hardy trace estimate recorded in \cite[(2.5)]{DaiWang2026}:
\begin{Eq}\label{eq:hardy-type-inequality}
\|r^{-1}f\|_{L_z^2}+\sup_{R>0}
\frac{\sinh^{(n-1)/2} R }{\sqrt{R}}\|f(R,\Theta)\|_{L_\Theta^2(\rS^{n-1})}
\lesssim \|\partial_rf+\frac{n-1}{2 \tanh r} f\|_{L_z^2}.
\end{Eq}
Since $\phi_{k-1}$ is spherically symmetric, this estimate applies to the radial functions below, and we obtain
\begin{Eq*}
I_{3}\lesssim&\kl\|r^{a-\frac{p}{2}}e^{-ar}\sinh^{\frac{(p-1)(n-1)}{2}}r\partial_t^{\leq 1}F_p(\partial_t\phi_{k-1})\kr\|_{L_{t,z}^2}\kl\|r^{\frac{p}{2}-a}e^{3ar}\sinh^{\frac{(p-1)(1-n)}{2}}r\kr\|_{L_r^\infty}\\
\lesssim&\kl\|r^{-\frac{1}{2}}\sinh^{\frac{n-1}{2}}r\partial_t\phi_{k-1}\kr\|_{L_{t,z}^\infty }^{p-1}\kl\|r^{a-\frac{1}{2}}e^{-ar}\partial_t^{\leq 1}\partial_t\phi_{k-1}\kr\|_{L_{t,z}^2},\\
\lesssim&\kl\|\tpartial\partial_t\phi_{k-1}\kr\|_{L_{t}^\infty L_z^2}^{p-1}\kl\|r^{a-\frac{1}{2}}e^{-ar}\tpartial\partial_t^{\leq 1}\phi_{k-1}\kr\|_{L_{t,z}^2}\\
\lesssim&\|\partial_t\phi_{k-1}\|_{E}^{p-1}\|\partial_t^{\leq 1}\phi_{k-1}\|_{LE}\lesssim \varE_{k-1}^p.
\end{Eq*}

In summary, we have shown that $\varE_k \lesssim \varepsilon + \varE_{k-1}^p$.
This implies that when $\varepsilon$ is sufficiently small, 
the sequence satisfies $\varE_k \lesssim \varepsilon$ uniformly in $k$.
For the contraction, we use the lower-order norm
\begin{Eq*}
\varF_k:=\|\phi_{k+1}-\phi_k\|_{E\cap LE}.
\end{Eq*}
The difference $\phi_{k+1}-\phi_k$ solves the same linear equation with zero initial data and forcing
$F_p(\partial_t\phi_k)-F_p(\partial_t\phi_{k-1})$.
Using the bound on $F_p'$ and the same Hardy estimate as above, we obtain
\begin{Eq*}
\|e^{2ar}(F_p(\partial_t\phi_k)-F_p(\partial_t\phi_{k-1}))\|_{L_{t,z}^2}
\lesssim(\varE_k+\varE_{k-1})^{p-1}\varF_{k-1}.
\end{Eq*}
Thus, for sufficiently small $\varepsilon$, $\varF_k\leq \varF_{k-1}/2$,
which in turn implies that the sequence $\{\phi_k\}$ converges to a solution $\phi$ of \eqref{eq:glassey-type-derivative-equation}.
A more detailed discussion of this standard lower-order contraction argument can be found in our previous work \cite{DaiWang2026}.
In particular, for the limiting function $\phi$, 
by applying the same type of estimate used for $I_3$,
we conclude that $\|e^{2ar}F_p(\partial_t\phi)\|_{L_{t,z}^2}<\infty$.
The same lower-order difference estimate also gives uniqueness in this solution class, and we omit the standard details.

The uniform bound on $\varE_k$ passes to the limit and, in particular, gives
\begin{Eq*}
\partial_t\phi\in L_t^\infty H^1,\qquad \partial_t^2\phi\in L_t^\infty L_z^2.
\end{Eq*}
Moreover, since $p<1+\frac{1}{n-2}<\frac{n}{n-2}$, the Sobolev embedding $H^1\subset L^{2p}$ gives
\begin{Eq*}
\|\partial_t\phi\|_{L_t^\infty L_z^{2p}}
\lesssim \|\partial_t\phi\|_{L_t^\infty H^1}<\infty.
\end{Eq*}
Hence $F_p(\partial_t\phi)\in L_t^\infty L_z^2$.  Using the equation, we therefore obtain
\begin{Eq*}
H_\kappa\phi=F_p(\partial_t\phi)-\partial_t^2\phi\in L_t^\infty L_z^2.
\end{Eq*}
This implies $\phi\in L_t^\infty H^2$.

Next, we prove ${\bphi}:=(\phi,\partial_t\phi)$ scatters in the space $\bH^{1}:=H^{1}\times L^2$.
For any $0<T_1 < T_2$, let $\phi_L(t)$ be the first component of the linear solution $L(t)\bphi(T_1)$, and set
\begin{Eq*}
w(t):=\phi(t+T_1)-\phi_L(t),\qquad 0\leq t\leq T_2-T_1.
\end{Eq*}
Then $w$ has zero initial data and forcing $F_p(\partial_t\phi(t+T_1))$.
By applying the energy identity to $L(T)$ and using \cref{prop:positive-mass-local-energy-estimate}, we obtain
\begin{Eq*}
&\|L(-T_2)\bphi(T_2)-L(-T_1)\bphi(T_1)\|_{\bH^{1}}
=\|\bphi(T_2)-L(T_2-T_1)\bphi(T_1)\|_{\bH^{1}}\\
\lesssim& \sup_{t\in[0,T_2-T_1]}\|\bphi(t+T_1)-L(t)\bphi(T_1)\|_{\bH^{1}}\\
\lesssim&\|w\|_{E\cap LE([0,T_2-T_1]\times M^{\circ})}\\
\lesssim&\|r^{a-\frac{1}{2}}e^{-ar}w\|_{L_{t,z}^2([0,T_2-T_1]\times M^{\circ})}
+\|e^{2ar}F_p(\partial_t\phi(t+T_1))\|_{L_{t,z}^2([0,T_2-T_1]\times M^{\circ})}.
\end{Eq*}
For the first term, we repeat the estimate for $I_2$ above: in $r\leq R_1$ it is absorbed into the $LE$ norm after fixing $R_1$ sufficiently small, while in $r\geq R_1$ it is controlled by the Strichartz estimate and hence by the forcing term.
As a result, we deduce that
\begin{Eq*}
\|L(-T_2)\bphi(T_2)-L(-T_1)\bphi(T_1)\|_{\bH^{1}}\lesssim&\|w\|_{E\cap LE([0,T_2-T_1]\times M^{\circ})}\\
\lesssim &\|e^{2ar}F_p(\partial_t\phi(t+T_1))\|_{L_{t,z}^2([0,T_2-T_1]\times M^{\circ})}\\
=&\|e^{2ar}F_p(\partial_t\phi)\|_{L_{t,z}^2([T_1,T_2]\times M^{\circ})}
\end{Eq*}

This means that
\begin{Eq*}
\sup_{T\geq 0}&\|L(-T)\bphi(T)\|_{\bH^{1}}\lesssim \|\bphi(0)\|_{\bH^{1}}+\sup_{T\geq 0}\|e^{2ar}F_p(\partial_t\phi)\|_{L_{t,z}^2([0,T]\times M^{\circ})}<\infty,\\
\lim_{T_1,T_2\rightarrow +\infty}&\|L(-T_2)\bphi(T_2)-L(-T_1)\bphi(T_1)\|_{\bH^{1}}\lesssim \lim_{T_1,T_2\rightarrow +\infty}\|e^{2ar}F_p(\partial_t\phi)\|_{L_{t,z}^2([T_1,T_2]\times M^{\circ})}=0.
\end{Eq*}
By the Cauchy criterion, 
we may define $(\psi_0^+,\psi_1^+):=\lim_{T\rightarrow+\infty} L(-T)\bphi(T)$, 
so that \eqref{eq:scattering-for-glassey-type-solution} holds. 
This completes the proof.

\subsection{Proof of Theorem~\ref{thm:smooth-nonlinearity-global-existence}}
In this theorem, 
the perturbative assumption \cref{hyp:small-asymptotic-decay} is available only in the exterior region,
while \cref{hyp:local-energy} is used in the region where $r$ is not large.
Thus we cannot apply \cref{prop:positive-mass-local-energy-estimate} directly on the original metric $g$.
To address this issue, we follow an approach inspired by the analysis on asymptotically Euclidean spaces (see, e.g., \cite{MetcalfeWang2017}),
and first derive a global local energy estimate adapted to this mixed framework.
For convenience, we denote $\square_g:=\partial_t^2-\Delta_g-(n-1)^2/4+\kappa^2$.
We now state the following result, whose proof is given in Appendix~\ref{subsec:proof-of-global-local-energy-estimate}.

\begin{proposition}\label{prop:global-local-energy-estimate}
Let $n \geq 3$, and assume that $\kappa>0$ and $(M^\circ, g)$ are as in \cref{thm:main-strichartz-estimate}.
In addition, suppose that $g$ can be rewritten in the $(r, \Theta)$ coordinate system and satisfies \cref{hyp:local-energy}, and satisfies \cref{hyp:small-asymptotic-decay} in the region $r>R_0$ for some sufficiently large $R_0$.
Then, for any fixed $0<2a<b<1$ and any integer $M\geq 0$, we have
\begin{Eq*}
\|Z^{\leq M}f\|_{[ E\cap  \widetilde{LE}]}
\lesssim& \|\tpartial^{\leq 1}Z^{\leq M} f(0)\|_{L_z^2}
+\|Z^{\leq M-1}\square_g f(0)\|_{L_z^2}\\
&+\|e^{2ar}Z^{\leq M}\square_g f\|_{L_{t,z}^2}
+\|(\partial_t,\nabla_\eta)^{\leq M-1}\square_g f\|_{(L_t^\infty\cap L_t^2)L_z^2(r\leq 2R_0)},
\end{Eq*}
for any smooth function $f$ whose spatial support is compact on each finite time interval, with the terms involving $M-1$ omitted when $M=0$ and the implicit constant depending on $a$ and $R_0$,
where
\begin{Eq*}
\|f\|_{ E}:=&\kl\|(\tpartial^{\leq 1},r^{-1}) f\kr\|_{L_t^\infty L_z^2}\approx \|\partial_tf\|_{L_t^\infty L_z^2}+\|f\|_{L_t^\infty H^1},\\
\|f\|_{ \widetilde{LE}}:=&\|(\partial_t,\nabla_\eta^{\leq 1})f\|_{L_t^2 L_z^2(r\lesssim R_0)}
+\|e^{-2a r}(\partial_t^{\leq 1},\d)f\|_{L_t^2 L_z^2(r\gtrsim R_0)}.
\end{Eq*}
\end{proposition}

\begin{remark}\label{rem:vector-field-conventions}
In \cref{prop:global-local-energy-estimate}, $
Z=\{\partial_t\}\cup\{K_i\}_{i=1}^n\cup\{\Omega_{ij}\}_{1\leq i<j\leq n}$,
where $K_i$ are the hyperbolic translations centered at the origin, and $\Omega_{ij}$ are the rotations on the standard hyperbolic space.
For detailed properties of these vector fields, we refer the reader to \cite{DaiWang2026}; the property used below is
\begin{Eq}\label{eq:hyperbolic-vector-field-identity}
\Delta_\eta=\sum_i K_i^2-\sum_{i<j}\Omega_{ij}^2,\qquad
\varL_Z\eta=0,\qquad [Z,\Delta_{\eta}]=0.
\end{Eq}
In the exterior region, with $\omega\in\rS^{n-1}$ and the standard convention $\Omega_{ij}=\omega_i\partial_{\omega_j}-\omega_j\partial_{\omega_i}$, we have
\begin{Eq*}
K_i=\omega_i\partial_r+\coth r\sum_j\omega_j\Omega_{ji}.
\end{Eq*}
For any fixed number of commutations, this identity and the tensor Lie derivative product rule show that the Lie derivatives of $g-\eta$ are controlled by \cref{hyp:small-asymptotic-decay}; the additional coefficient-derivative terms are bounded in the $\eta$-norm in the exterior region.
The norm $|\d f|$ in $\widetilde{LE}$ is measured as the $\eta$-norm of the spatial covector $\d f$ in the fixed $(r,\Theta)$ coordinate framework.
The derivatives $(\partial_t,\nabla_\eta)^{\leq M-1}$ in the compact forcing term may be replaced, on $\{r\leq 2R_0\}$, by any equivalent fixed finite family of local derivatives.
\end{remark}

\begin{remark}
The local energy hypothesis \cref{hyp:local-energy} contains the lower-order remainder
$\|e^{-ar}\phi\|_{L_{t,z}^2}$.
This term does not appear in \cref{prop:global-local-energy-estimate} because, in the proof of the proposition,
the corresponding contribution is localized to the interior region and is controlled by the global Strichartz estimate \cref{thm:main-strichartz-estimate}.
Thus it is absorbed in the final estimate used in the nonlinear iteration.
\end{remark}

To begin with, we construct a sequence by setting $\phi_{-1} := 0$ and defining $\phi_k$ for $k \geq 0$ to satisfy
\begin{Eq*}
\begin{cases}
\square_g \phi_k=
F_2(\phi_{k-1},\partial_t\phi_{k-1})
 +F_0(\phi_{k-1},\partial_t\phi_{k-1})(\d\phi_{k-1},\d\phi_{k-1})_g; \\ (\phi_k,\partial_t\phi_k)_{t=0}=\varepsilon(\psi_0,\psi_1).
\end{cases}
\end{Eq*}
We will show that the sequence $\{\phi_k\}$ converges, 
thereby yielding a solution to \eqref{eq:smooth-nonlinearity-equation}.
Here, we first carry out the iteration at one fixed high Sobolev level.
To proceed, fix a sufficiently large integer $M$, once and for all, so that the Sobolev-Moser estimates on fixed compact regions and the exterior Sobolev estimates used below are valid.
The constants in the estimates may depend on this fixed $M$, and the smallness threshold for $\varepsilon$ is chosen afterwards.
After the finite-order solution is obtained, smoothness for smooth compactly supported data follows by the standard persistence-of-regularity argument, applying the same type of higher-order estimates successively, as in our previous work \cite{DaiWang2026}.
We introduce the notations
\begin{Eq*}
\varE_k:=\|Z^{\leq M}\phi_k\|_{E\cap \widetilde{LE}},\qquad
\varF_k:=\|\phi_{k+1}-\phi_k\|_{E\cap \widetilde{LE}},
\end{Eq*}
where $E$ and $\widetilde{LE}$ are those introduced in \cref{prop:global-local-energy-estimate}.
The notation and basic properties of $Z$ are recorded in Remark~\ref{rem:vector-field-conventions}.
The proof has two steps.
First, we prove that the iterates are uniformly bounded in the high norm, namely $\varE_k\lesssim\varepsilon$.
Then we prove that the iteration is contractive in the base norm $\varF_k$.

We now proceed with the iteration argument.
First, observe that $\varE_{-1} = 0\lesssim\varepsilon$. 
Next, we assume that for some $k$, we have $\varE_{k-1} \lesssim \varepsilon$.
We choose $a>0$ sufficiently small so that $0<2a<b$ and $4a<(n-1)/2$.
Applying \cref{prop:global-local-energy-estimate} to $\phi_k$ on $[0,T]$, with constants independent of $T$, and then letting $T\to\infty$, we obtain
\begin{Eq*}
\varE_k\lesssim& \|\tpartial^{\leq 1}Z^{\leq M}\phi_k(0)\|_{L^2}
+\|Z^{\leq M-1}\big(F_2(\phi_{k-1},\partial_t\phi_{k-1})
+F_0(\phi_{k-1},\partial_t\phi_{k-1})(\d\phi_{k-1},\d\phi_{k-1})_g\big)(0)\|_{L_z^2}\\
&\qquad+\|e^{2ar}Z^{\leq M}F_2(\phi_{k-1},\partial_t\phi_{k-1})\|_{L_{t,z}^2}\\
&\qquad+\|e^{2ar}Z^{\leq M}\big(F_0(\phi_{k-1},\partial_t\phi_{k-1})(\d\phi_{k-1},\d\phi_{k-1})_g\big)\|_{L_{t,z}^2}\\
&\qquad+\|(\partial_t,\nabla_\eta)^{\leq M-1}F_2(\phi_{k-1},\partial_t\phi_{k-1})\|_{(L_t^\infty\cap L_t^2)L_z^2(r\leq 2R_0)}\\
&\qquad+\|(\partial_t,\nabla_\eta)^{\leq M-1}\big(F_0(\phi_{k-1},\partial_t\phi_{k-1})(\d\phi_{k-1},\d\phi_{k-1})_g\big)\|_{(L_t^\infty\cap L_t^2)L_z^2(r\leq 2R_0)}\\
\equiv&I_1+I_2+I_3+I_4+I_5+I_6.
\end{Eq*}
For $I_1$ and $I_2$, following the same reasoning as in \cite{DaiWang2026}, we have $I_1+I_2 \lesssim \varepsilon$, and we omit the details here.
We next estimate the nonlinear contributions $I_3,\ldots,I_6$.
Set $\mathcal D f:=(f,\partial_t f,\d f)$.
With the convention in \cref{prop:global-local-energy-estimate}, for any scalar function $f$ we first have, in the compact region,
$|\d f|_\eta\lesssim |\nabla_\eta f|$.
In the exterior region, the spatial exterior derivative can be estimated by the hyperbolic Killing fields:
\begin{Eq*}
|\d f|_\eta
\lesssim |\partial_r f|+\sinh^{-1}r\,|\nabla_{\Theta} f|
\lesssim |Kf|+\sinh^{-1}r\,|\Omega f|
\lesssim |Zf|.
\end{Eq*}
In particular, we only use the hyperbolic angular factor $\sinh^{-1}r\leq 1$ in the exterior region.
When $Z$ acts on the covector $\d f$, we use the Lie derivative convention, so that $\varL_Z(\d f)=\d(Zf)$.
Consequently, by the Sobolev embedding on the fixed compact region, the exterior spherical Sobolev estimate, \eqref{eq:hardy-type-inequality}, and the preceding bound for $\d f$, we have
\begin{Eq*}
\|\mathcal D Z^{\leq \lfloor M/2\rfloor+1}\phi_{k-1}\|_{L_{t,z}^{\infty}(r\leq 2R_0)}
+\|r^{-\frac12}e^{\frac{n-1}{2}r}
\mathcal D Z^{\leq \lfloor M/2\rfloor+1}\phi_{k-1}\|_{L_{t,z}^{\infty}(r\geq R_0)}
\lesssim \varE_{k-1}.
\end{Eq*}

For $I_3$, using the smoothness of $F_2$ and the vanishing conditions $F_2(0,0)=DF_2(0,0)=0$,
the chain rule and the Sobolev-Moser estimates show that
$Z^{\leq M}F_2(\phi_{k-1},\partial_t\phi_{k-1})$ is a finite sum of terms of the schematic form
\begin{Eq*}
A\big(\partial_t^{\leq 1}Z^{\leq \lfloor M/2\rfloor+1}\phi_{k-1}\big)\,
\big(\partial_t^{\leq 1}Z^{\leq \lfloor M/2\rfloor+1}\phi_{k-1}\big)\,
\big(\partial_t^{\leq 1}Z^{\leq M}\phi_{k-1}\big),
\end{Eq*}
where $A$ is a bounded smooth coefficient depending on finitely many derivatives of $F_2$, as long as $\varE_{k-1}$ is controlled from above.
The extra low-order factor reflects the absence of constant and linear terms in $F_2$.
Thus it suffices to estimate one such product.
In the compact region $r\lesssim R_0$, using the compact part of the preceding pointwise bound, we obtain
\begin{Eq*}
I_{3;(r\lesssim R_0)}
\lesssim C(\varE_{k-1})\varE_{k-1}\|(\partial_t^{\leq1},\nabla_\eta)Z^{\leq M}\phi_{k-1}\|_{L_{t,z}^2(r\lesssim R_0)}
\lesssim C(\varE_{k-1})\varE_{k-1}^2.
\end{Eq*}

For the exterior region $r\gtrsim R_0$, using the exterior part of the preceding pointwise bound and taking $a>0$ sufficiently small, for instance $4a<(n-1)/2$, we have
\begin{Eq*}
I_{3;(r\gtrsim R_0)}
\lesssim& C(\varE_{k-1})
\|r^{-\frac{1}{2}}e^{\frac{n-1}{2}r}\mathcal D Z^{\leq \lfloor M/2\rfloor+1}\phi_{k-1}\|_{L_{t,z}^\infty(r\gtrsim R_0)}
\\
&\qquad \times\|e^{-2ar}\mathcal D Z^{\leq M}\phi_{k-1}\|_{L_{t,z}^2(r\gtrsim R_0)} 
\cdot \|r^{\frac{1}{2}}e^{(4a-\frac{n-1}{2})r}\|_{L^\infty}\\
\lesssim& C(\varE_{k-1})\varE_{k-1}^2.
\end{Eq*}
Combining the two regions yields $I_3\lesssim C(\varE_{k-1})\varE_{k-1}^2$.

We now turn to $I_4$.
By the product rule, the chain rule, and the identity $\varL_Z(\d f)=\d(Zf)$,
the differentiated nonlinearity in $I_4$ is a finite sum of terms of the following two schematic types:
\begin{Eq*}
&A_1\big(\partial_t^{\leq 1}Z^{\leq \lfloor M/2\rfloor+1}\phi_{k-1}\big)
\big[\varL_Z^{\leq M}g^{-1}\big]\big(\d Z^{\leq \lfloor M/2\rfloor+1}\phi_{k-1},
\d Z^{\leq M}\phi_{k-1}\big)\\
&\quad+
A_2\big(\partial_t^{\leq 1}Z^{\leq \lfloor M/2\rfloor+1}\phi_{k-1}\big)
\big(\partial_t^{\leq 1}Z^{\leq M}\phi_{k-1}\big)
\big[\varL_Z^{\leq \lfloor M/2\rfloor+1}g^{-1}\big]\big(\d Z^{\leq \lfloor M/2\rfloor+1}\phi_{k-1},
\d Z^{\leq \lfloor M/2\rfloor+1}\phi_{k-1}\big).
\end{Eq*}
The first type corresponds to the case where the highest derivative falls on the metric contraction or on one of the exterior derivatives, while the second type corresponds to the case where the highest derivative falls on $F_0(\phi_{k-1},\partial_t\phi_{k-1})$.
We denote their corresponding contributions to $I_4$ by $I_{41}$ and $I_{42}$, respectively.
Here $A_1$ and $A_2$ denote bounded smooth coefficients depending on finitely many derivatives of $F_0$, including the lower-order factors produced by differentiating $F_0(\phi_{k-1},\partial_t\phi_{k-1})$.
On the compact region, the factors involving $\varL_Zg^{-1}$ are bounded by the smoothness of $g$.
In the exterior region, since $g^{-1}=\eta^{-1}+O(e^{-br})$ with the corresponding differentiated bounds from \cref{hyp:small-asymptotic-decay}, and since $Z$ is Killing for $\eta$, these metric factors are bounded, and are perturbative when a derivative lands on $g^{-1}-\eta^{-1}$.
Using the compact part of the pointwise bound and the uniform comparability of $g$ and $\eta$ in the fixed compact region, we get
\begin{Eq*}
I_{41;(r\lesssim R_0)}
\lesssim C(\varE_{k-1})
\|\d Z^{\leq M}\phi_{k-1}\|_{L_{t,z}^2(r\lesssim R_0)}
\|\mathcal D Z^{\leq \lfloor M/2\rfloor+1}\phi_{k-1}\|_{L_{t,z}^\infty(r\lesssim R_0)}
\lesssim C(\varE_{k-1})\varE_{k-1}^2.
\end{Eq*}
For the exterior region, using \cref{hyp:small-asymptotic-decay} again to compare the $g$- and $\eta$-norms, and taking $a>0$ sufficiently small, for instance $4a<(n-1)/2$, we obtain
\begin{Eq*}
I_{41;(r\gtrsim R_0)}
\lesssim& C(\varE_{k-1})
\|r^{-\frac{1}{2}}e^{\frac{n-1}{2}r}\d Z^{\leq \lfloor M/2\rfloor+1}\phi_{k-1}\|_{L_{t,z}^\infty(r\gtrsim R_0)}
\\
&\qquad\times
\|e^{-2ar}\d Z^{\leq M}\phi_{k-1}\|_{L_{t,z}^2(r\gtrsim R_0)}
\cdot \|r^{\frac{1}{2}}e^{(4a-\frac{n-1}{2})r}\|_{L^\infty}
\lesssim C(\varE_{k-1})\varE_{k-1}^2.
\end{Eq*}
Thus $I_{41}\lesssim C(\varE_{k-1})\varE_{k-1}^2$.
For $I_{42}$, the factor $\partial_t^{\leq 1}Z^{\leq M}\phi_{k-1}$ is estimated as in $I_3$, while the lower-order metric contraction is controlled by the same pointwise bounds used for $I_{41}$.
Therefore,
\begin{Eq*}
I_{42}\lesssim C(\varE_{k-1})\varE_{k-1}^2.
\end{Eq*}
Thus $I_4\lesssim C(\varE_{k-1})\varE_{k-1}^2$.

It remains to estimate the compact forcing terms $I_5$ and $I_6$.
On the fixed compact region $r\leq 2R_0$, the derivatives $(\partial_t,\nabla_\eta)^{\leq M-1}$ are controlled by finitely many $Z$-derivatives of comparable order.
For the $L_t^2L_z^2$ part, we use the compact local energy component of $\varE_{k-1}$.
For the $L_t^\infty L_z^2$ part, the same compact Sobolev-Moser estimate is combined with the energy component of $\varE_{k-1}$.
By the vanishing of $F_2$ to second order and the quadratic structure of the second nonlinearity, the compact parts of the estimates for $I_3$ and $I_4$ give
\begin{Eq*}
I_5+I_6\lesssim C_{R_0}C(\varE_{k-1})\varE_{k-1}^2.
\end{Eq*}

At this point, we have obtained the uniform inequality
\begin{Eq*}
\varE_k\leq C_0\kl(\varepsilon+C(\varE_{k-1})\varE_{k-1}^2\kr).
\end{Eq*}
Therefore, if $\varE_{k-1}\leq 2C_0\varepsilon$ and $\varepsilon$ is sufficiently small, then $\varE_k\leq 2C_0\varepsilon$.
By induction, we have $\varE_k\lesssim\varepsilon$ uniformly in $k$.
For the contraction, the difference equation has forcing given by the difference between
\begin{Eq*}
F_2(\phi_k,\partial_t\phi_k)
+F_0(\phi_k,\partial_t\phi_k)(\d\phi_k,\d\phi_k)_g
\end{Eq*}
and the same expression with $\phi_k$ replaced by $\phi_{k-1}$.
By applying the zeroth-order version of \cref{prop:global-local-energy-estimate} to the difference equation, and using the mean value theorem together with the same lower-order nonlinear estimates, exactly as in the corresponding contraction argument in \cite{DaiWang2026}, one obtains
$\varF_k\leq C_1\varepsilon\,\varF_{k-1}$.
Taking $\varepsilon$ sufficiently small, we have $\varF_k\leq \varF_{k-1}/2$.
Hence $\{\phi_k\}$ converges to a solution $\phi\in H^{M}$.
This completes the proof.

\section{Appendix}\label{sec:appendix}

In this appendix, we prove the two local energy estimates used in the nonlinear applications, namely \cref{prop:positive-mass-local-energy-estimate} and \cref{prop:global-local-energy-estimate}.

\subsection{Proof of Proposition~\ref{prop:positive-mass-local-energy-estimate}}\label{subsec:proof-of-positive-mass-local-energy-estimate}

\begin{proof}
We prove the estimate on a finite time interval $[0,T]$, with constants independent of $T$; the global estimate follows by letting $T\to\infty$.
The idea of the proof is essentially the local energy argument from our earlier work \cite{DaiWang2026}; for the reader's convenience, we include a reasonably complete sketch and indicate where the positive mass term enters.
The perturbative calculation in \cite{DaiWang2026} is carried out for a more general class of time-dependent second-order perturbations; in the present static setting the only new feature relevant to the multiplier identity is the positive mass term.
For the unperturbed hyperbolic operator, write
\begin{Eq*}
\tilde\square_{\rH}:=-\nabla^\alpha\nabla_\alpha+c,
\qquad
c=\kappa^2-\frac{(n-1)^2}{4}.
\end{Eq*}
Since $\kappa>0$, we have $c>-\frac{(n-1)^2}{4}$.
Throughout the proof, we set
\begin{Eq*}
w=\frac{n-1}{2\tanh r}.
\end{Eq*}
We also use the radial divergence notation
\begin{Eq*}
\hpartial_r f:=\partial_r f+2wf=\sinh^{1-n}r\,\partial_r(\sinh^{n-1}r f).
\end{Eq*}
We first prove the estimate for the unperturbed operator $\tilde\square_{\rH}$.  Let $\mathfrak m=-dt^2+\eta$ be the Lorentzian product metric used in this unperturbed computation; indices in this part of the proof are raised using $\mathfrak m$, and $\nabla$ denotes its Levi-Civita connection.
The energy-momentum tensor is
\begin{Eq*}
\tilde T_0^{\alpha\beta}
:=\nabla^\alpha f\nabla^\beta f
-\frac12 \mathfrak m^{\alpha\beta}\kl(\nabla_\sigma f\nabla^\sigma f+cf^2\kr).
\end{Eq*}
For a multiplier triple $(X,\chi,Y)$, define
\begin{Eq*}
\tilde J_0^\alpha
:=-\tilde T_0^{\alpha\beta}X_\beta
+\frac12 f^2\nabla^\alpha\chi-\chi f\nabla^\alpha f+Y^\alpha .
\end{Eq*}
A direct computation gives
\begin{Eq*}
\nabla_\alpha\tilde J_0^\alpha
=\tilde\square_{\rH}f(Xf+\chi f)
-\tilde T_0^{\alpha\beta}\pi_{\alpha\beta}
+\frac12f^2\nabla_\alpha\nabla^\alpha\chi
-\chi\nabla_\alpha f\nabla^\alpha f
-c\chi f^2+\nabla_\alpha Y^\alpha .
\end{Eq*}

Taking $(X,\chi,Y)=(\partial_t,0,0)$, we obtain the conserved energy density
\begin{Eq*}
\tilde J_0^t
=\frac12\kl(
|\partial_t f|^2+|\partial_r f+wf|^2+|\spartial f|^2
+\kl(\frac{(n-1)^2}{4}+c+\frac{(n-1)(n-3)}{\sinh^2r}\kr)f^2
\kr)
-\frac12\hpartial_r(wf^2).
\end{Eq*}
Since $c>-(n-1)^2/4$ and the divergence term integrates to zero, this controls the energy norm.

We next use a radial multiplier.  Set
\begin{Eq*}
(X,\chi,Y)=\kl(h\partial_r,hw,-\frac12h'\tw f^2\partial_r\kr),
\end{Eq*}
where $\tw=w$ or $\tw=w-r^{-1}$.  The boundary term is the same as in the shifted wave case and is bounded by the energy.  For the divergence term one computes
\begin{Eq*}
\nabla_\alpha\tilde J_0^\alpha
=&h\tilde\square_{\rH}f(\partial_r f+wf)
-\frac{h'}2\kl(|\partial_t f|^2+|\partial_r f+\tw f|^2\kr)
-\kl(\frac{h}{\tanh r}-\frac{h'}2\kr)|\spartial f|^2\\
&-\frac12Wf^2
+\frac{h'}2\kl(\frac{(n-1)^2}{4}+c\kr)f^2,
\end{Eq*}
where
\begin{Eq*}
W=h''(\tw-w)
+h'\kl(-2w'+\tw'-2w^2+2w\tw-\tw^2+\frac{(n-1)^2}{4}\kr)
-h(w''+2ww').
\end{Eq*}
For the given $a\in(0,1)$ choose
\begin{Eq*}
h(r)=1+\int_0^r a\rho^{2a-1}e^{-2a\rho}\,d\rho .
\end{Eq*}
If $a<1/2$, this multiplier is first replaced by a smooth regularization near $r=0$; the estimates below are uniform in the regularization parameter, and we then pass to the limit.
Then $h$ is bounded and
\begin{Eq*}
h'\approx r^{2a-1}e^{-2ar},\qquad
\frac{h}{\tanh r}-\frac{h'}2\gtrsim h'.
\end{Eq*}
Using the identities
\begin{Eq*}
w^2+w'-\frac{(n-1)^2}{4}=\frac{(n-1)(n-3)}{4\sinh^2r},
\qquad
w''+2ww'=-\frac{2}{\tanh r}\frac{(n-1)(n-3)}{4\sinh^2r},
\end{Eq*}
we obtain, with $\tw=w$,
\begin{Eq*}
W\gtrsim \frac{n-3}{\sinh^2r\tanh r},
\end{Eq*}
and with $\tw=w-r^{-1}$,
\begin{Eq*}
W\gtrsim r^{2a-3}e^{-2ar}.
\end{Eq*}
The only term that cannot be absorbed is
\begin{Eq*}
\frac{h'}2\kappa^2 f^2,
\end{Eq*}
which is precisely controlled by
\begin{Eq*}
\|r^{a-\frac12}e^{-ar}f\|_{L_{t,z}^2}^2.
\end{Eq*}
This is the additional zeroth-order local term caused by the positive mass.
After integration over spacetime and Cauchy's inequality, this yields
\begin{Eq*}
\|f\|_{E\cap LE}^2
\lesssim
\|\tilde\square_{\rH}f\|_{L_t^1L_z^2+LE^*}
\|f\|_{E\cap LE}
+\|r^{a-\frac12}e^{-ar}f\|_{L_{t,z}^2}^2
+\|\tpartial^{\leq1}f(0)\|_{L^2}^2 .
\end{Eq*}
This gives the desired estimate for the unperturbed operator.

It remains to include the perturbative terms.  Under \cref{hyp:small-asymptotic-decay}, the difference between $P_g$ and $\tilde\square_{\rH}$ has coefficients decaying like $e^{-br}$ with $b>2a$.  In the multiplier identity, the typical errors have the form
\begin{Eq*}
O(\varepsilon'e^{-br})\nabla f\,\nabla f,\qquad
O(\varepsilon'e^{-br})f\,\nabla f,
\end{Eq*}
after integrating by parts in the second-order perturbative terms.  These weights are stronger than the squared local energy weight $r^{2a-1}e^{-2ar}$ because $b>2a$, and hence the total perturbative contribution is bounded by
\begin{Eq*}
\varepsilon' \|f\|_{E\cap LE}^2 .
\end{Eq*}
For sufficiently small $\varepsilon'$, these terms are absorbed into the left-hand side.  Replacing $\tilde\square_{\rH}f$ by $P_g f$ then proves the desired estimate.
\end{proof}

\begin{remark}
The proof above is intentionally kept close to the argument in \cite{DaiWang2026}.  If one wants to apply the estimate to $\partial_t f$ for time-dependent coefficients, commutator terms must be included.  In the static geometric applications in the present paper this causes no additional difficulty; a fully time-dependent statement should record the corresponding commutator bounds explicitly.
\end{remark}

\subsection{Proof of Proposition~\ref{prop:global-local-energy-estimate}}\label{subsec:proof-of-global-local-energy-estimate}

The estimate follows from the zeroth-order estimate in Lemma~\ref{lem:zeroth-order-global-local-energy-estimate} by the standard higher-order localized-energy argument, as in the proof of \cite[Proposition~1.4, Section~3.3]{Wang2015}.
Indeed, applying Lemma~\ref{lem:zeroth-order-global-local-energy-estimate} to $Z^\alpha f$, $|\alpha|\leq M$, and summing over $\alpha$, gives the desired high-order estimate except for the commutators between $Z^\alpha$ and $\square_g$.
In the exterior region these commutators are perturbative: since the vector fields $Z$ commute with the model operator $\square_\eta$, the principal commutator coefficients come from $g-\eta$.
By \cref{hyp:small-asymptotic-decay}, these coefficients decay like $e^{-br}$, and after multiplication by the weight $e^{2ar}$ they are bounded by $e^{-(b-2a)r}$, hence can be absorbed for $R_0$ sufficiently large since $0<2a<b$.
In the compact region, no smallness is assumed.
The commutators are handled by the localized elliptic regularity and induction-on-commutators argument used in \cite[Proposition 1.4]{Wang2015}; the worst second spatial derivatives of $Z^{\leq M-1}f$ are converted by the equation into time derivatives already included in the high-order local energy norm, differentiated forcing terms, and lower commutators.
The latter are controlled by induction on the number of commutations.
This argument produces precisely the additional compact forcing terms
\begin{Eq*}
\|Z^{\leq M-1}\square_g f(0)\|_{L_z^2}
+\|(\partial_t,\nabla_\eta)^{\leq M-1}\square_g f\|_{(L_t^\infty\cap L_t^2)L_z^2(r\leq 2R_0)}
\end{Eq*}
appearing in the statement of \cref{prop:global-local-energy-estimate}; these are the analogues of the compact forcing terms in \cite[Proposition~1.4, Section~3.3]{Wang2015}.
It remains only to prove the zeroth-order estimate.

\begin{lemma}[Zeroth-order estimate]\label{lem:zeroth-order-global-local-energy-estimate}
Let $n \geq 3$, and assume that $\kappa>0$ and $(M^\circ, g)$ are as in \cref{thm:main-strichartz-estimate}.
In addition, suppose that $g$ can be rewritten in the $(r, \Theta)$ coordinate system and satisfies \cref{hyp:local-energy}, and satisfies \cref{hyp:small-asymptotic-decay} in the region $r>R_0$ for some sufficiently large $R_0$.
Then, for any fixed $0<2a<b<1$, we have
\begin{Eq*}
\|f\|_{[ E\cap  \widetilde{LE}]}\lesssim& \|\tpartial^{\leq 1} f(0)\|_{L_z^2}+\|e^{2ar}\square_g f\|_{L_{t,z}^2},
\end{Eq*}
for any smooth function $f$ whose spatial support is compact on each finite time interval, where $E$ and $\widetilde{LE}$ are defined in \cref{prop:global-local-energy-estimate}.
\end{lemma}

For any constant $s > 0$, define $\chi_s(r) = \chi(r/s)$, where $\chi$ is defined in \cref{sec:preparation}. 
This construction ensures that $\chi_s$ equals $1$ on the interval $(0,s]$ and vanishes on $[2s, \infty)$.
For convenience, we denote $\square_g:=\partial_t^2-\Delta_g-(n-1)^2/4+\kappa^2$.

\begin{proof}
First, observe that if $R_0$ satisfies the condition in \cref{prop:global-local-energy-estimate}, then any $cR_0$ with $c \geq 1$ also satisfies it.
We now divide the argument into cases depending on the value of $r$, using $4R_0$ as the threshold, where $R_0$ is the fixed constant appearing in \cref{prop:global-local-energy-estimate}.
This interior-exterior decomposition, involving both the function and the metric, follows the standard strategy used in \cite{HMSSZ}.

\noindent  {\bf Case 1: Local energy estimate in the region $r\leq 4R_0$.}
To apply \cref{hyp:local-energy}, we first isolate the contribution generated by the data and forcing localized near this region.
For this purpose, we construct a function $f_1$ solving the equation
\begin{Eq*}
\square_g f_1=\chi_{4R_0}\cdot \square_gf,\qquad (f_1,\partial_t f_1)_{t=0}=\chi_{4R_0}\cdot(f,\partial_t f)_{t=0}.
\end{Eq*}
By applying \cref{hyp:local-energy} with a radius comparable to $R_0$, along with the same H\"older-Strichartz argument used to estimate $I_2$ in the proof of \cref{thm:glassey-type-global-existence} in \cref{subsec:proof-of-glassey-type-global-existence}, we obtain
\begin{Eq*}
\|(\partial_t,\nabla_\eta^{\leq 1})f_1\|_{L_{t,z}^2(r\leq 4R_0)}
\lesssim&\|(\partial_t,\nabla_g^{\leq 1})f_1\|_{L_{t,z}^2(r\leq 4R_0)}\\
\lesssim& \|\chi_{4R_0}f(0)\|_{H^1}+\|\chi_{4R_0}\partial_tf(0)\|_{L_z^2}+\|\chi_{4R_0}\square_g f\|_{L_{t,z}^2}+\|e^{-ar}f_1\|_{L_{t,z}^2}\\
\lesssim&\|\tpartial^{\leq 1} f(0)\|_{L_z^2}+\|e^{2ar}\square_g f\|_{L_{t,z}^2}.
\end{Eq*}
In the last step, the lower-order term $\|e^{-ar}f_1\|_{L_{t,z}^2}$ is controlled by the global Strichartz estimate \cref{thm:main-strichartz-estimate} through that same H\"older-Strichartz argument: on the compact region this follows from H\"older's inequality and the Strichartz bound, while in the exterior region the exponential weight is integrable and the same bound applies.
Meanwhile, let $f_2 = f - f_1$. Then $f_2$ satisfies
\begin{Eq*}
\square_g f_2=(1-\chi_{4R_0})\cdot \square_gf,\qquad (f_2,\partial_t f_2)_{t=0}=(1-\chi_{4R_0})\cdot(f,\partial_t f)_{t=0}.
\end{Eq*}

To handle the exterior contribution, we replace $g$ by a metric which agrees with $\eta$ in the interior and with $g$ in the exterior.  Define
\begin{Eq*}
\tg=\chi_{R_0} \eta+(1-\chi_{R_0})g=\eta+(1-\chi_{R_0})\gamma.
\end{Eq*}
Since $R_0$ is chosen sufficiently large, this metric satisfies the assumptions of \cref{prop:positive-mass-local-energy-estimate}.
Indeed, $\tg-\eta=(1-\chi_{R_0})(g-\eta)$, and the transition-region derivatives are still controlled by the original $e^{-br}$ decay; by increasing $R_0$, these perturbative constants are made sufficiently small.
In particular, after increasing $R_0$ if necessary, $\tg$ is a sufficiently small, exponentially decaying global perturbation of $\eta$ and hence also satisfies the geometric and spectral assumptions of \cref{thm:main-strichartz-estimate}.  The lower-order term in \cref{prop:positive-mass-local-energy-estimate} for the glued metric is therefore treated by the same H\"older-Strichartz argument used for $I_2$ in the proof of \cref{thm:glassey-type-global-existence}.
We then consider the function $f_{21} := (1 - \chi_{2R_0}) f_{21'}$, where $f_{21'}$ solves
\begin{Eq*}
\square_{\tg} f_{21'}=(1-\chi_{4R_0})\cdot \square_gf,\qquad (f_{21'},\partial_t f_{21'})_{t=0}=(1-\chi_{4R_0})\cdot(f,\partial_t f)_{t=0}.
\end{Eq*}
Applying \cref{prop:positive-mass-local-energy-estimate} to this equation, together with the same H\"older-Strichartz control of the lower-order term, we obtain
\begin{Eq*}
\|(\partial_t,\nabla_\eta^{\leq 1})f_{21}\|_{L_{t,z}^2(r\leq 4R_0)}\lesssim&\|(\partial_t,\nabla_\eta^{\leq 1})f_{21'}\|_{L_{t,z}^2(r\leq 4R_0)}\lesssim\|(\tpartial,r^{-1})f_{21'}\|_{L_{t,z}^2}\\
\lesssim&\|\tpartial^{\leq 1} f_{21'}(0)\|_{L_z^2}+\|r^{a-\frac{1}{2}}e^{-a r}f_{21'}\|_{L_{t,z}^2}+\| r^{\frac{1}{2}-a}e^{ar}\square_{\tg} f_{21'}\|_{L_{t,z}^2}\\
\lesssim&\|\tpartial^{\leq 1} f(0)\|_{L_z^2}+\| e^{2ar}\square_{\tg} f_{21'}\|_{L_{t,z}^2}\\
\lesssim&\|\tpartial^{\leq 1} f(0)\|_{L_z^2}+\| e^{2ar}\square_{g} f\|_{L_{t,z}^2}.
\end{Eq*}

Finally, let $f_{22} = f_2 - f_{21}$. 
Since $(g - \tg)$ vanishes in the region $r \geq 2R_0$ and $\supp f_{21} \subset \{r \geq 2R_0\}$,
it follows that $(f_{22},\partial_t f_{22})_{t=0}=(0,0)$ and
\begin{Eq*}
\square_g f_{22}=&\square_g f_2-\square_{\tg}\big((1-\chi_{2R_0})f_{21'}\big)\\
=&(1-\chi_{4R_0})\cdot \square_gf-(1-\chi_{2R_0})\square_{\tg} f_{21'}+[\chi_{2R_0},\Delta_{\tg}]f_{21'}\\
=&[\chi_{2R_0},\Delta_{\tg}]f_{21'}.
\end{Eq*}
Note that $\supp [\chi_{2R_0},\Delta_{\tg}]f_{21'}\subset\{r\leq 4R_0\}$. 
On this fixed region, the local energy weights are comparable to constants, so the local unweighted norms below are controlled by the corresponding interior local energy norms.
By once again applying \cref{hyp:local-energy} with a radius comparable to $R_0$, along with a similar argument as before, we obtain
\begin{Eq*}
\|(\partial_t,\nabla_\eta^{\leq 1})f_{22}\|_{L_{t,z}^2(r\leq 4R_0)}
\lesssim&\|(\partial_t,\nabla_g^{\leq 1})f_{22}\|_{L_{t,z}^2(r\leq 4R_0)}\\
\lesssim&\|[\chi_{2R_0},\Delta_{\tg}]f_{21'}\|_{L_{t,z}^2}
\lesssim\|(\partial_t,\nabla_\eta^{\leq 1})f_{21'}\|_{L_{t,z}^2(r\leq 4R_0)}\\
\lesssim&\|\tpartial^{\leq 1} f(0)\|_{L_z^2}+\| e^{2ar}\square_{g} f\|_{L_{t,z}^2}.
\end{Eq*}
In summary, we complete the proof of the local energy estimate in the region $r \leq 4R_0$.

\noindent  {\bf Case 2: Local energy estimate in the region $r\geq 4R_0$.}
In this case, we directly define $f_3 := (1 - \chi_{2R_0}) f$, which satisfies
\begin{Eq*}
\square_{\tg}f_3=\square_g f_3=[\square_g,\chi_{2R_0}]f+(1-\chi_{2R_0})\square_gf,\qquad (f_3,\partial_t f_3)_{t=0}=(1-\chi_{2R_0})\cdot(f,\partial_t f)_{t=0}.
\end{Eq*}
Since $(f - f_3)$ vanishes in the region $r \geq 4R_0$, 
we may apply \cref{prop:positive-mass-local-energy-estimate}, 
together with the same H\"older-Strichartz argument used for $I_2$ in \cref{subsec:proof-of-glassey-type-global-existence}, 
and the result from Case 1, to obtain
\begin{Eq*}
\|e^{-2ar}(\partial_t^{\leq 1},\d)f\|_{L_{t,z}^2(r\geq 4R_0)}\lesssim& \|r^{a-\frac{1}{2}}e^{-ar}(\tpartial,r^{-1})f_3\|_{L_{t,z}^2}
+\|r^{a-\frac{1}{2}}e^{-ar}f_3\|_{L_{t,z}^2}\\
\lesssim&\|\tpartial^{\leq 1}f_3(0)\|_{L_{z}^2}+\|r^{a-\frac{1}{2}}e^{-ar}f_3\|_{L_{t,z}^2}+\|r^{\frac{1}{2}-a}e^{ar}\square_{\tg}f_3\|_{L_{t,z}^2}\\
\lesssim&\|\tpartial^{\leq 1}f(0)\|_{L_{z}^2}+\|e^{2ar}\square_{g}f\|_{L_{t,z}^2}+\|[\square_g,\chi_{2R_0}]f\|_{L_{t,z}^2}\\
\lesssim&\|\tpartial^{\leq 1}f(0)\|_{L_{z}^2}+\|e^{2ar}\square_{g}f\|_{L_{t,z}^2}. 
\end{Eq*}
Here we used that, in the exterior region,
\begin{Eq*}
|\d f|_\eta
\lesssim |\partial_r f|+|\spartial f|
\lesssim |\partial_r f+w f|+|\spartial f|+|f|,
\end{Eq*}
so the exterior $\d f$ term is controlled by the derivative part of $\tpartial f$ together with the displayed lower-order $f_3$ term.
For $r\geq 4R_0$, we have $e^{-2ar}\lesssim r^{a-\frac12}e^{-ar}$, so the exterior weight in $\widetilde{LE}$ is controlled by the local energy weight in \cref{prop:positive-mass-local-energy-estimate}.
This establishes the local energy estimate in the region $r \geq 4R_0$.

\noindent  {\bf Case 3: Energy estimate}. 
By applying the standard energy method, we obtain
\begin{Eq*}
A(T)=&A(0)+\int_0^T (\partial_t f,\square_gf)_g\d t,\\
 A(t):=&\frac{(\partial_t f,\partial_t f)_g+(\sqrt{H_\kappa}f,\sqrt{H_\kappa }f)_g}{2}\approx \|\partial_t f\|_{L_z^2}^2+\|f\|_{H^1}^2\approx\|(\tpartial^{\leq 1},r^{-1})f\|_{L_z^2}^2.
\end{Eq*}
Thus, using the results of the local energy estimate, we obtain
\begin{Eq*}
\|f\|_{E}^2\lesssim& A(0)+\|e^{-2ar}\partial_tf\|_{L_{t,z}^2}\cdot\|e^{2ar}\square_g f\|_{L_{t,z}^2}\\
\lesssim&\|\tpartial^{\leq 1}f(0)\|_{L_{z}^2}^2+\|e^{2ar}\square_{g}f\|_{L_{t,z}^2}^2.
\end{Eq*}
This gives the desired energy estimate. 

Combining the conclusions from all three cases, we complete the proof of Lemma~\ref{lem:zeroth-order-global-local-energy-estimate}.
\end{proof}

 \bibliography{bib.bib}

\end{document}